\documentclass[graybox]{svmult}

\usepackage{appendix}
\usepackage{mathptmx}       
\usepackage{helvet}         
\usepackage{courier}        
\usepackage{type1cm}        
\usepackage{bm}
\usepackage{amsmath}
\usepackage{amssymb}
\usepackage{amsfonts}

\usepackage{makeidx}         
\usepackage{graphicx}        
\usepackage{multicol}        
\usepackage[bottom]{footmisc}
\usepackage{algorithmic}
\usepackage[ruled]{algorithm}
\usepackage {dsfont}

\RequirePackage{amsmath}
\RequirePackage{amssymb}

\usepackage[unicode=true,plainpages=false]{hyperref}
\hypersetup{colorlinks=false,pdfborder={0 0 0.1},linkcolor=magenta,anchorcolor=magenta,urlcolor=blue,citecolor=blue,
          pdftitle={Computation of the Response Surface in the Tensor Train data format},
          pdfauthor={Sergey Dolgov, Boris N. Khoromskij, Alexander Litvinenko and Hermann G. Matthies}
          }

\usepackage{tikz} \usetikzlibrary{positioning,shapes}
\usepackage{pgfplots} \usetikzlibrary{plotmarks}
\usetikzlibrary{calc,backgrounds}
\input{pgfart.sty}

\DeclareMathAlphabet{\mathitbf}{OML}{cmm}{b}{it}
\def\jmath{j}

\newcommand{\dotprod}[2]{ \left \langle #1, #2 \right \rangle }

\def\eps{\varepsilon}
\def\diag{\mathop{\rm diag}\nolimits}
\def\supp{\mathop{\rm supp}\nolimits}

\def\nin{\hspace{0.22em}/ \hspace{-0.88em} \in}       

\def\u{\mathbf{u}}
\def\f{\mathbf{f}}
\def\K{\mathbf{K}}
\def\BDelta{\bm{\Delta}}
\def\J{\mathcal{J}}

\def\ttp{\mathtt{p}}

\makeindex

\newcommand{\Nset}{\mathbb{N}}
\newcommand{\cov}{\mbox{cov}}
\newcommand{\var}{\mbox{var}}

\def\thetab{{\theta}}

\begin{document}

\title*{Computation of the Response Surface in the Tensor Train data format}
\author{Sergey Dolgov, Boris N. Khoromskij, Alexander Litvinenko and Hermann G. Matthies}
\institute{H.G. Matthies (\email{wire@tu-bs.de}) \at Institute for Scientific Computing, Technische Universit\"at Braunschweig, Hans-Sommerstr. 65, Brunswick, Germany,
\and  A. Litvinenko (\email{alexander.litvinenko@kaust.edu.sa}) \at King Abdullah University of Science and Technology (KAUST), Thuwal, Saudi Arabia,
\and Sergey Dolgov (\email{dolgov@mis.mpg.de}) and Boris N. Khoromskij (\email{bokh@mis.mpg.de})  \at Max-Planck-Institut f\"ur Mathematik
in den Naturwissenschaften, 
Inselstra\ss e 22,
04103 Leipzig, Germany.}

%
\maketitle

\abstract{
We apply the Tensor Train (TT) approximation to construct the Polynomial Chaos Expansion (PCE) of a random field, and solve the stochastic elliptic diffusion PDE with the stochastic Galerkin discretization.
We compare two strategies of the polynomial chaos expansion: sparse and full polynomial (multi-index) sets.
In the full set, the polynomial orders are chosen independently in each variable, which provides higher flexibility and accuracy.
However, the total amount of degrees of freedom grows exponentially with the number of stochastic coordinates.
To cope with this curse of dimensionality, the data is kept compressed in the TT decomposition, a recurrent low-rank factorization.
PCE computations on sparse grids sets are extensively studied, but the TT representation for PCE is a novel approach that is investigated in this paper.
We outline how to deduce the PCE from the covariance matrix, assemble the Galerkin operator, and evaluate some post-processing (mean, variance, Sobol indices), staying within the low-rank framework.
The most demanding are two stages.
First, we interpolate PCE coefficients in the TT format using a few number of samples, which is performed via the block cross approximation method.
Second, we solve the discretized equation (large linear system) via the alternating minimal energy algorithm.
In the numerical experiments we demonstrate that the full expansion set encapsulated in the TT format is indeed preferable in cases when high accuracy and high polynomial orders are required.
}

\section{Motivation}
\label{sec:1}

During the last years, low-rank tensor techniques were successfully
applied to the solution of high-dimensional stochastic and parametric
PDEs \cite{doostan-non-intrusive-2013, khos-pde-2010, khos-qtt-2010, KhSch-Galerkin-SPDE-2011, tobkress-param-2011, tobler-htdmrg-2011, matthieszander-lowrank-2012, Nouy10, Nouy07, litvinenko-spde-2013, wahnert-stochgalerkin-2014, Grasedyck_Kressner,   kunoth-2013, LitvSampling13}.
With standard techniques it is almost
impossible to store all entries of the discretised high-dimensional
operator explicitly. Besides of the storage one should solve this
high-dimensional problem and obtain solution in a reasonable time. Very often some additional efficient post-processing
is required, e.g. visualization of large data sets, computation of the mean, variance, Sobol indices or exceedance probabilities
of the solution or of a quantity of interest (QoI). We will perform this post-processing in the low-rank tensor format.

Situations where one is concerned with uncertainty quantification often
come in the following guise: we are investigating some physical system which
is modeled as follows:
\begin{equation} \label{eq:general_model}
F(\ttp;u(x,\omega)) = f(\ttp;x;\omega),
\end{equation}
where $u(\cdot)$ describes the state of the system
lying in a Hilbert space $\mathcal{V}$,
$F$ is an operator, modeling the physics of the system, and $f$ is some external
influence (action / excitation / loading).
In particular, $F(\ttp;u(\cdot))$ could be some parameter-dependent differential operator, for example,
\begin{equation}
\label{eq:elliptic}
-\nabla \cdot (\kappa(x,\omega) \nabla u(x,\omega))= f(x,\omega), \quad x \in D\subset \mathbb{R}^d,
\end{equation}
where $d$ is the spatial dimension, $\ttp:=\kappa(x,\omega)$ is a random field dependent on a random
parameter $\omega$ in some probability space $\Omega$, and $\mathcal{V} = L_2(D)$.
In this article we approximate/represent the input coefficient $\kappa(x,\omega)$ in the low-rank tensor format, compute the solution $u(x,\omega)$ and perform all post-processing in the same low-rank data format.
The model \eqref{eq:general_model} depends on some parameter $\ttp \in \mathcal{P}$; in the context of uncertainty quantification the actual value of $\ttp$ is undefined \cite{UQLitvinenko12, Rosic2013}.
Often this uncertainty is modeled by giving the set $\mathcal{P}$ a probability measure.
Evaluation and quantification of the uncertainty often requires functionals of the state $\varPsi(u(\ttp))$, and the functional dependence of $u$ on $\ttp$ becomes important.
Similar situations arise in design, where $\ttp$ may be a design parameter still to be chosen, and one may seek a design such that a functional $\varPsi(u(\ttp))$ is e.g. maximized.

The situation just sketched involves a number of objects which are functions
of the parameter values.  While evaluating $F(\ttp)$ or $f(\ttp)$ for a certain
$\ttp$ may be straightforward, one may easily envisage situations where evaluating
$u(\ttp)$ or $\Psi(u(\ttp))$ may be very costly as it may rely on some very time
consuming simulation or computation, like, for example, running a climate model.

As it will be shown in the next section, any such
parametric object like $u(\ttp)$, $F(\ttp)$, or $f(\ttp)$ may be seen as an element
of a tensor product space \cite{litvinenko-spde-2013, wahnert-stochgalerkin-2014}.
This can be used to find very sparse approximations to those objects, and hence much cheaper ways to evaluate parametric quantities in the uncertainty quantification, like means, covariances, exceedance probabilities, etc.
For this purpose, the dependence of $F(\ttp)$ and $f(\ttp)$ on $\ttp$ has to be propagated to the solution or state vector
$u(\ttp)$, see e.g. \cite{matthies-galerkin-2005}.

The main theoretical result of this work is approximation of the response surface of a stochastic model in the tensor train (TT) data format.
In the next section we outline the general Galerkin, PCE and KLE discretization schemes for a random field.
The introduction to the TT methods, and the new block cross interpolation algorithm are presented in Section \ref{sec:TTlow}.
A bit more details how to apply the block cross algorithm to the response surface based on a multivariate polynomial (polynomial chaos expansion, PCE) approximation see in Section \ref{sec:tt-kle-pce}.
We start with the TT approximation of the multi-dimensional input coefficient $\kappa$.
After that, in Section \ref{sec:stiff_tt} we construct the stochastic Galerkin matrix in the TT format.
Section \ref{subsec:levelSets} is devoted to the efficient post-processing (computation of level sets, Sobol indices (Section \ref{sec:SA}), the mean value and covariance) in the TT format.
Numerical results in Section \ref{sec:numerics} demonstrate the performance of the TT format, applied to the approximate solution of the elliptic boundary problem with uncertain coefficients.
The notation list is given in Table \ref{tab:notations}.

\begin{table}[h]
\caption{Notation}
\begin{tabular}{|c|l|}
\hline \hline
\multicolumn{2}{|c|}{\textbf{General quantities}} \\ \hline
${\Nset}$        &Natural numbers \\ \hline
$\f$            &The right hand side \\ \hline
$\u$          & The solution \\ \hline
$D$ & Computational domain \\ \hline
$\mathcal{P}$  & Space where parameter $p$ is living \\ \hline
$\mathcal{V}_N$ & vector space spanned on the basis $\{\varphi_1(x),\ldots,\varphi_N(x)\}$ \\ \hline
$I$, $I_N$ & Identical operator, identity matrix of size $N\times N$ \\ \hline \hline
\multicolumn{2}{|c|}{\textbf{Stochastic quantities and expansions}} \\ \hline
{KL}           & {Karhunen-Lo\`eve (Expansion)}\\ \hline
gPCE            & generalized Polynomial Chaos Expansion \\ \hline
${\cov}(x,y)$          &Covariance function\\ \hline
${\var}$           & Variance\\ \hline
$\gamma(x,\omega)$ & Gaussian random field  \\ \hline
$\kappa(x,\omega)$ & Non-Gaussian random field, $\kappa(x,\cdot)=\phi(\gamma(x,\cdot))$ \\ \hline
$\phi$ & Transformation, $\kappa=\phi(\gamma)$ \\ \hline
$\kappa_{\alpha}(x)$ & PCE coefficient at the index $\alpha$ \\ \hline
$\bar{\kappa}(x)$ & The mean value  \\ \hline
$H_{\alpha}$ & Multivariate Hermite polynomial with multi-index $\alpha:=(\alpha_1,\ldots,\alpha_M)$ \\ \hline
$h_{\alpha_m}(\theta_m)$ & univariate polynomial of variable $\theta_m$ \\ \hline
$\BDelta_{\alpha,\beta,\nu}$  & Tensor product of the Hermite polynomial algebra $\Delta$-matrices   \\ \hline
$\Delta_{\alpha_1,\beta_1,\nu_1} $ & A single $\Delta$ matrix defined for single indices $\alpha_1,\beta_1,\nu_1$ \\ \hline
$K_0,\ldots,K_L$ & Galerkin matrices with coefficient $v_l(x)$, $l=0,\ldots,L$ \\ \hline \hline
\multicolumn{2}{|c|}{\textbf{Indices, ranges and sets}} \\ \hline
$\mathbf{i}$ & Multi-index $\mathbf{i}:=(i_1,\ldots,i_{d})$\\ \hline
$\alpha$, $\beta$, $\nu$ & stochastic multi-indices  \\ \hline
$\J$, $\J_{M,p}$ & Infinite and finite-dimensional multi-index sets \\ \hline
$\J_{M,p}^{sp}$ & Sparse multi-index set \\ \hline
$L$ & Number of the KLE terms after truncation \\ \hline
$M$ & Stochastic dimension \\ \hline
$d$ & Physical (spatial) dimension \\ \hline
$p$ & polynomial order or vector of them, $p=(p_1,\ldots,p_M)$ \\ \hline \hline
\multicolumn{2}{|c|}{\textbf{Multidimensional and/or tensor product quantities}} \\ \hline
{TT}           & Tensor train data format \\ \hline
$\otimes$  & Tensor (Kronecker) product \\ \hline
$\Join $ & strong Kronecker product \\ \hline
$u(\theta_1,\ldots,\theta_M)$ & $M$-dimensional function \\ \hline
$\mathbb{I}^{(k)}$ & left quasi-maxvol indices, $\{\alpha_1,\ldots,\alpha_{k}\}\in\mathbb{I}^{(k)}$ \\ \hline
$\mathbb{J}^{(k)}$ & right quasi-maxvol indices, $\{\alpha_{k+1},\ldots,\alpha_{M}\} \in \mathbb{J}^{(k)}$ \\ \hline
\end{tabular}
\label{tab:notations}
\end{table}%

\section{Discretisation and computation}
\label{SS:discret}
For brevity we follow \cite{Matthies_encycl}, where more references may be found,
see also the recent monograph \cite{LeMatre10}.
Usually \eqref{eq:general_model} is some partial differential equation and has to be discretised,
approximated, or somehow projected onto a finite dimensional subspace
\begin{equation}
\mathcal{V}_N = \mathrm{span}\{\varphi_1(x),\ldots,\varphi_N(x)\} \subset \mathcal{V}, \quad \mbox{with} \quad \dim \mathcal{V}_N = N.
\label{eq:V_N}
\end{equation}
For example, a set of piecewise linear finite elements may be considered.

To propagate the parametric dependence, choose a finite dimensional
subspace of the Hilbert space,
say $\mathcal{P}_J \subset \mathcal{P}$ for the solution $u(\ttp)$
in \eqref{eq:general_model}.
Via the Galerkin projection or collocation, or other such techniques,
the parametric model is thereby formulated on
the tensor product $\mathcal{V}_N \otimes \mathcal{P}_J$, denoted as
\begin{equation}  \label{eq:all_discrete}
\mathbf{F}(\u) = \f.
\end{equation}
For example, the linear elliptic equation \eqref{eq:elliptic} is cast to the linear system $\K \u = \f$.
There are multiple possibilities for the choice of $\mathcal{P}$, and hence finite dimensional subspaces $\mathcal{P}_J$.
The solution of \eqref{eq:all_discrete} is often computationally challenging,
as $\dim (\mathcal{V}_N \otimes \mathcal{P}_J) = N \cdot J$ may be very large.
One way to handle high-dimensional problems is to apply a low-rank
approximation, by representing
the entities in \eqref{eq:all_discrete}, e.g. $\mathbf{F}$, $\u$, and
$\f$ in a low-rank data format.  Several numerical techniques
\cite{Nouy2009, Doostan2009, Krosche_2010, matthieszander-lowrank-2012}
have been developed recently to obtain an approximation to the
solution of \eqref{eq:all_discrete} in the form $\u(x,\ttp) \approx \sum_j \mathitbf{u}_j(x) \cdot \mathitbf{z}_j(\ttp)$.
It is important that such techniques operate only with the elements of the data-sparse low-rank representation, and hence solve the high-dimensional problem efficiently \cite{KhorSurvey, Schwab_sparse_tensors11, ost-latensor-2009, ot-tt-2009, tobkress-param-2011, litvinenko-spde-2013, Sudret_sparsePCE}.

Once this has been computed, any other functional like $\varPsi(u(\ttp))$ may be found with relative ease.
As soon as $\mathcal{P}$ is equipped with a probability measure, the functionals usually take the form of expectations, a variance, an exceedance probability, or other statistic quantities needed in the uncertainty quantification.


\subsection{Discretization of the input random field}

We assume that $\kappa(x,\omega)$ may be seen as a smooth transformation $\kappa = \phi(\gamma)$ of the Gaussian random field $\gamma(x,\omega)$.
In this Section we explain how to compute KLE of $\gamma$ if KLE of $\kappa$ is given. For more details see \cite[Section 3.4.2]{ZanderDiss} or \cite{KeeseDiss}.

Typical example is the log-normal field with $\phi(\gamma) = \exp(\gamma)$.
The Gaussian distribution of $\gamma$ stipulates to decompose $\phi$ into the Hermite polynomial series,
\begin{equation}
\phi(\gamma) = \sum\limits_{i=0}^{\infty} \phi_i h_i(\gamma) \approx \sum\limits_{i=0}^{Q} \phi_i h_i(\gamma), \quad \phi_i = \int\limits_{-\infty}^{+\infty} \phi(z) \frac{1}{i!} h_i(z) \exp(-z^2/2) dz,
\label{eq:phi_tranform}
\end{equation}
where $h_i(z)$ is the $i$-th Hermite polynomial, and $Q$ is the number of terms after truncation.

The Gaussian field $\gamma(x,\omega)$ may be written as the Karhunen-Loeve expansion (KLE).
First, given the covariance matrix of $\kappa(x,\omega)$, we may relate it with the covariance matrix of $\gamma(x,\omega)$ as follows,
\begin{equation}
\cov_{\kappa}(x,y) = \int_{\Omega} \left(\kappa(x,\omega)-\bar{\kappa}(x)\right) \left(\kappa(y,\omega)-\bar{\kappa}(y)\right) dP(\omega) \approx \sum_{i=0}^{Q} i! \phi_i^2 \cov_{\gamma}^i (x,y).
\end{equation}
Solving this implicit $Q$-order equation, we derive $\cov_{\gamma} (x,y)$.
Now, the KLE may be computed,
\begin{equation}
\gamma(x,\omega) = \sum_{m=1}^{\infty} g_m(x) \theta_m(\omega), \quad \mbox{where} \quad \int\limits_{D} \cov_{\gamma}(x,y) g_m(y)dy = \lambda_m g_m(x),
\end{equation}
and we assume that the eigenfunctions $g_m$ absorb the square roots of the KL eigenvalues, such that the stochastic variables $\theta_m$ are normalized (they are uncorrelated and jointly Gaussian).

In the discrete setting, we truncate PCE and write it for $M$ random variables,
\begin{equation}
\kappa(x,\omega) \approx \sum\limits_{\alpha \in \J_M} \kappa_{\alpha}(x) H_{\alpha}(\theta(\omega)), \quad \text{where}\quad H_{\alpha}(\theta):=h_{\alpha_1}(\theta_1) \cdots h_{\alpha_M}(\theta_M)
\label{eq:pce_k}
\end{equation}
is the multivariate Hermite polynomial, $\alpha = (\alpha_1,\ldots,\alpha_M)$ is the multi-index of the PCE coefficients, $h_{\alpha_m}(\theta_m)$ is the univariate Hermite polynomial,
and $\J_M = \bigotimes_{m=1}^M \{0,1,2,\ldots\}$ is a set of all multinomial orders (or a multi-index set).
The Galerkin coefficients $\kappa_{\alpha}$ are evaluated as follows,
\begin{equation}
\kappa_{\alpha}(x) = \dfrac{(\alpha_1+\cdots+\alpha_M)!}{\alpha_1! \cdots \alpha_M! } \phi_{\alpha_1+\cdots+\alpha_M} \prod_{m=1}^{M} g_m^{\alpha_m}(x),
\label{eq:pce_k_coeffs}
\end{equation}
where $\phi_{|\alpha |}:= \phi_{\alpha_1+\cdots+\alpha_M}$ is the Galerkin coefficient of the transform function in \eqref{eq:phi_tranform},
and $g_m^{\alpha_m}(x)$ means just the $\alpha_m$-th power of the KLE function value $g_m(x)$.
The expansion \eqref{eq:pce_k} still contains infinite amount of terms.
In practice, we restrict the polynomial orders to finite limits,
which can be done in different ways.
In the current paper, we investigate the following two possibilities.

\begin{definition}
The \emph{full} multi-index is defined by restricting each component independently,
$$
\J_{M,p} = \{0,1,\ldots,p_1\} \otimes \cdots \otimes \{0,1,\ldots,p_M\}, \quad \mbox{where} \quad p=(p_1,\ldots,p_M)
$$
is a shortcut for the tuple of order limits.
\end{definition}

The full set provides high flexibility in resolution of stochastic variables
\cite{wahnert-stochgalerkin-2014, litvinenko-spde-2013}.
However, its cardinality is equal to $\prod_{m=1}^M (p_m+1) \le (p+1)^M$, if $p_m\le p$.
This may become a very large number even if $p$ and $M$ are moderate ($p \sim 3$, $M \sim 20$ is typical).
In this paper, we do not store all $(p+1)^M$ values explicitly, but instead approximate them via the \emph{tensor product} representation. See more about multi-indices in Sec. 3.2.1 in \cite{ZanderDiss}.

A traditional way to get rid of the curse of dimensionality is the sparse expansion set.
\begin{definition}
The \emph{sparse} multi-index is defined by restricting the sum of components,
$$
\J_{M,p}^{sp} = \left\{\alpha=(\alpha_1,\ldots,\alpha_M):\quad \alpha \ge 0,~~\alpha_1+\cdots+\alpha_M \le p \right\}.
$$
\end{definition}

The sparse set contains $\mathcal{O}\left(\frac{M!}{p!(M-p)!}\right) = \mathcal{O}(M^p)$ values if $M \gg p$,
which is definitely much less than $(p+1)^M$.
However, the negative side is that for a fixed $p$, some variables are worse resolved than the others, and the approximation accuracy may suffer.
It may be harmful to increase $p$ either, since in the sparse set it contributes exponentially to the complexity.

The tensor product storage of the full coefficient set scales as $\mathcal{O}(M p r^2)$, where $r$ is the specific measure of the ``data structure''.
Typically, it grows with $M$ and $p$, fortunately in a mild way (say, linearly).
In cases when $M$ is moderate, but $p$ is relatively large, this approach may become preferable.
Other important features are the adaptivity to the particular data, and easy assembly of the stiffness matrix of the PDE problem.

Some complexity reduction in Formula \eqref{eq:pce_k_coeffs} can be achieved with the help of the KLE for the initial field $\kappa(x,\omega)$.
Consider the expansion
\begin{equation}
\kappa(x,\omega) = \bar\kappa(x) + \sum_{\ell=1}^{\infty} \sqrt{\mu_\ell} v_\ell(x) \eta_\ell(\omega)
\approx  \bar\kappa(x) + \sum_{\ell=1}^{L} \sqrt{\mu_\ell} v_\ell(x) \eta_\ell(\omega),
\label{eq:k_kle}
\end{equation}
where $v_\ell(x)$ are eigenfunctions of the integral operator with the covariance as the kernel.
The set $\{v_\ell(x)\}$ can be used as a (reduced) basis in $L_2(D)$.
It is difficult to work with \eqref{eq:k_kle} straightforwardly, since the distribution of $\eta_\ell$ is generally unknown.
But we know that the set $V(x) = \{v_\ell(x)\}_{\ell=1}^{L}$, where $L$ is the number of KLE terms after the truncation, serves as an optimal reduced basis.
Therefore, instead of using \eqref{eq:pce_k_coeffs} directly, we project it onto $V(x)$:
\begin{equation}
\tilde \kappa_{\alpha}(\ell) = \dfrac{(\alpha_1+\cdots+\alpha_M)!}{\alpha_1! \cdots \alpha_M! } \phi_{\alpha_1+\cdots+\alpha_M} \int\limits_{D} \prod_{m=1}^{M} g_m^{\alpha_m}(x) v_\ell(x) dx.
\label{eq:k_kle_pce_coeff}
\end{equation}
Note that the range $\ell=1,\ldots,L$ may be much smaller than the number of $x$-discretization points $N$.
After that, we restore the approximate coefficients,
\begin{equation}
\kappa_{\alpha}(x) \approx \bar\kappa(x) + \sum\limits_{\ell=1}^{L} v_\ell(x) \tilde \kappa_{\alpha}(\ell).
\label{eq:k_kle_pce}
\end{equation}

\subsection{Construction of the stochastic Galerkin operator}
The same PCE ansatz of the coefficient \eqref{eq:pce_k} may be adopted to discretize the solution $u$, the test functions from the stochastic Galerkin method, and ultimately the whole initial problem \eqref{eq:general_model}, see \cite{wahnert-stochgalerkin-2014, litvinenko-spde-2013}.
To be concrete, we stick to the elliptic equation \eqref{eq:elliptic} with the deterministic right-hand side $f=f(x)$.

Given the KLE components \eqref{eq:k_kle} and the spatial discretization basis \eqref{eq:V_N},
we first assemble the spatial Galerkin matrices,
\begin{equation}
K_0(i,j) = \int\limits_{D} \bar\kappa(x) \nabla \varphi_i(x) \cdot \nabla \varphi_j(x) dx, \quad K_\ell(i,j) = \int\limits_{D} v_\ell(x) \nabla \varphi_i(x) \cdot \nabla \varphi_j(x) dx,
\label{eq:K_l}
\end{equation}
for $i,j=1,\ldots,N$, $\ell=1,\ldots,L$.
Now we take into account the PCE part $\tilde \kappa_{\alpha}$.
Assuming that $u$ is decomposed in the same way as \eqref{eq:pce_k} with the same $\J_{M,p}$ or $\J_{M,p}^{sp}$, we integrate over stochastic coordinates $\theta$ and obtain the following rule (see also \cite{KeeseDiss, ZanderDiss, Matthies_encycl}),
\begin{equation}
K_{\alpha,\beta}(\ell) = \int\limits_{\mathbb{R}^M} H_{\alpha}(\theta) H_{\beta}(\theta) \sum\limits_{\nu \in \J_{M,p}} \tilde\kappa_{\nu}(\ell) H_{\nu}(\theta) d\theta = \sum\limits_{\nu \in \J_{M,p}} \BDelta_{\alpha,\beta,\nu} \tilde\kappa_{\nu}(\ell),
\label{eq:K_p}
\end{equation}
where
\begin{equation}
\BDelta_{\alpha,\beta,\nu} =  \Delta_{\alpha_1,\beta_1,\nu_1} \cdots \Delta_{\alpha_M,\beta_M,\nu_M}, \quad  \Delta_{\alpha_m,\beta_m,\nu_m} =  \int\limits_{\mathbb{R}} h_{\alpha_m}(z) h_{\beta_m}(z) h_{\nu_m}(z) dz,
\label{eq:Delta}
\end{equation}
is the triple product of the Hermite polynomials, and $\tilde\kappa_{\nu}(\ell) $ is according to \eqref{eq:k_kle_pce_coeff}.
For convenience, since $\Delta$ is a super-symmetric tensor, we will denote as $\Delta_{\nu_m} \in\mathbb{R}^{(p+1) \times (p+1)}$ a matrix slice at the fixed index $\nu_m$ (or the whole multi-index $\nu$ in $\BDelta$).
Putting together \eqref{eq:k_kle_pce}, \eqref{eq:K_l} and \eqref{eq:K_p}, we obtain the whole discrete stochastic Galerkin operator,
\begin{equation}
\K = K_0 \otimes \BDelta_0 + \sum\limits_{\ell=1}^L K_\ell \otimes \sum\limits_{\nu \in \J_{M,p}} \BDelta_{\nu} \tilde \kappa_{\nu}(\ell),
\label{eq:K_gen}
\end{equation}
which is a square matrix of size $N \cdot \#\J_{M,p} = N(p+1)^M$ in case of full $\J_{M,p}$.
Fortunately, if $\tilde \kappa_{\nu}$ is computed in the tensor product format, the direct product in $\BDelta$ \eqref{eq:Delta} allows to exploit the same format for \eqref{eq:K_gen}, and build the operator easily.

Since the only stochastic input is the permeability $\kappa$, the right-hand side is extended to the size of $\K$ easily,
\begin{equation}
\f = f_0 \otimes e_{0}, \qquad f_0(i) = \int\limits_{D} \varphi_i(x) f(x)dx, \quad i=1,\ldots,N,
\label{eq:f}
\end{equation}
and $e_0$ is the first identity vector of size $\#\J_{M,p}$, $e_0 = (1,0,\ldots,0)^\top$, which assigns the deterministic $f(x)$ to the zeroth-order Hermite polynomial in parametric domain.

\section{Tensor product formats and low-rank data compression}
\label{sec:TTlow}
\subsection{Tensor Train decomposition}
We saw that the cardinality of the full polynomial set $\J_{M,p}$ may rise to prohibitively large values $(p+1)^M$, and the traditional way to get rid of this curse of dimensionality is the sparsification $\J_{M,p}^{sp}$.
We propose to use the full set, but approximate the data indirectly in the low-rank tensor product format,
which may be more flexible than the prescribed sparsification rule.

To show the techniques in the most brief way, we choose the so-called \emph{matrix product states} (MPS) formalism~\cite{fannes-mps-1992}, which introduces the following representation of a multi-variate tensor:
\begin{equation}\label{eq:tt}
 \begin{split}
 u(\alpha)  & =  \tau(u^{(1)},\ldots,u^{(M)}), \quad\mbox{meaning}\\
  u(\alpha_1,\ldots,\alpha_M)
   & = \sum_{s_1=1}^{r_1} \sum_{s_2=1}^{r_2}  \cdots \sum_{s_{M-1}=1}^{r_{M-1}}
      u^{(1)}_{s_1}(\alpha_1) u^{(2)}_{s_1,s_2}(\alpha_2) \cdots u^{(M)}_{s_{M-1}}(\alpha_M), \quad \mbox{or} \\
  u(\alpha_1,\ldots,\alpha_M) & = u^{(1)}(\alpha_1) u^{(2)}(\alpha_2) \cdots u^{(M)}(\alpha_M), \quad \mbox{or}\\
  u & = \sum_{s_1=1}^{r_1} \sum_{s_2=1}^{r_2}  \cdots \sum_{s_{M-1}=1}^{r_{M-1}}
      u^{(1)}_{s_1} \otimes u^{(2)}_{s_1,s_2} \otimes \cdots \otimes u^{(M)}_{s_{M-1}}.
 \end{split}
\end{equation}
In numerical linear algebra this format became known under the name~\emph{tensor train} (TT) representation since \cite{ot-tt-2009,osel-tt-2011}.
Each TT core (or \emph{block}) $u^{(k)}=[u^{(k)}_{s_{k-1},s_k}(\alpha_k)]$ is defined by $r_{k-1}(p_k+1) r_k$ numbers, where $p_k$ denotes the number of basis functions (i.e. the polynomial order) in the variable $\alpha_k$ (the \emph{mode size}), and $r_k=r_k(u)$ is the \emph{TT rank}.
The total number of entries scales as $\mathcal{O}(Mpr^2)$, which is tractable as long as $r=\max\{r_k\}$ is moderate.

Each line of the definition \eqref{eq:tt} is convenient for its own purposes.
The first one is a formal statement that a tensor $u$ is presented in the TT format, which may be used if we change one block, for example.
The second line is the most expanded elementwise definition, showing the polylinearity of the format.
The third line is the seminal Matrix Product form: after fixing $\alpha_k$, each block $u^{(k)}(\alpha_k)$ has only two indices $s_{k-1},s_k$, and may be considered as a matrix.
Thus, the TT format means that each element of $u$ is presented as a product of matrices.
Finally, the fourth line is an analog of the \emph{Canonical Polyadic} decomposition, convenient to compare the structures in both formats.
Surely, in the same form we may write e.g. $\kappa_{\alpha} = \kappa^{(1)}(\alpha_1) \cdots \kappa^{(M)}(\alpha_M)$.

\subsection{Analytical construction of the TT format}

High-dimensional operators (cf. $\K$ in \eqref{eq:K_gen}) may be naturally presented in the TT format, by putting two indices in each TT block,  $\K^{(k)} = [\K^{(k)}_{s_{k-1},s_k}(\alpha_k,\beta_k)]$.
\begin{example}
Consider the $M$-dimensional Laplacian discretized over an uniform tensor grid ($n$ dofs in each direction).
It has the Kronecker (canonical) rank-$M$ representation:
\begin{equation}
\mathbf{A} = A\otimes I \otimes \cdots \otimes I + I \otimes A \otimes  \cdots \otimes I  + \cdots + I\otimes I \otimes  \cdots \otimes A \in \mathbb{R}^{n^M \times n^M}
\end{equation}
with $A= \mbox{tridiag}\{-1,2,-1\}\in \mathbb{R}^{n \times n}$, and $I$ being the $n \times n$ identity.
However, the same operator in the TT format is explicitly representable with all TT ranks equal to 2 for any dimension \cite{khkaz-lap-2012},
\begin{equation}
\mathbf{A}
 = (A \; I)\Join
 \begin{pmatrix}
  I & 0 \\
  A & I
 \end{pmatrix}\Join
...
 \Join
 \begin{pmatrix}
  I & 0 \\
  A & I
 \end{pmatrix}\Join
 \begin{pmatrix}
  I  \\
  A
 \end{pmatrix},
 \end{equation}
where the strong Kronecker product operation "$\Join $" is defined as a regular matrix product for the first level of TT cores, and the inner blocks (e.g. $A$, $I$) are multiplied by means of the Kronecker product.
In the elementwise matrix product notation, the Laplace operator reads
$$
\mathbf{A}(\mathbf{i},\mathbf{j}) =
\begin{pmatrix}A(i_1,j_1) & I(i_1,j_1)\end{pmatrix}
\begin{pmatrix}
  I(i_2,j_2) & 0 \\
  A(i_2,j_2) & I(i_2,j_2)
 \end{pmatrix} \cdots
 \begin{pmatrix}
  I(i_d,j_d)  \\
  A(i_d,j_d)
 \end{pmatrix}.
$$
\end{example}

\subsection{Approximate construction of the TT format}

The principal favor of the TT format comparing to e.g. the Canonical Polyadic (CP) decomposition (which is also popular in statistics, see \cite{litvinenko-spde-2013, wahnert-stochgalerkin-2014}) is a stable quasi-optimal rank reduction procedure \cite{osel-tt-2011}, based on singular value decompositions.
The complexity scales as $\mathcal{O}(Mpr^3)$, i.e. is free from the curse of dimensionality,
while the full accuracy control takes place.

But before we outline this procedure, we need to introduce some notation.
In the TT definition \eqref{eq:tt}, we see that TT blocks depend on 3 indices, and in this sense are 3-dimensional tensors.
However, when we write computational tensor algorithms, it is convenient (and efficient in practice!) to cast tensor contractions to (sequences of) matrix products.
One way to deduce a matrix from a 3-dimensional block is to take a slice, fixing one index, which was used in the Matrix Product form of \eqref{eq:tt}.
Another approach is the conception of \emph{reshaping}, i.e. when two of three indices are considered as the new multiindex, such that all block elements become written in a matrix.
In the programming practice, this operation may be explicit, like the \texttt{reshape} function in MATLAB, or implicit, like in a Fortran+LAPACK program, where all data is passed simply in a contiguous array, and the sizes are specified as separate inputs.

In mathematics, this issue may be formalized via the following special symbols.
\begin{definition}\label{def:folded}
Given a TT block $u^{(k)}$, introduce the following reshapes:
\begin{itemize}
 \item \emph{left-folded block}: $u^{|k\rangle}(\overline{s_{k-1},\alpha_k},s_k) = u^{(k)}_{s_{k-1},s_k}(\alpha_k), \quad u^{|k\rangle} \in \mathbb{R}^{r_{k-1}(p_k+1) \times r_k},$ \quad and
 \item \emph{right-folded block}: $u^{\langle k |}(s_{k-1},\overline{\alpha_k,s_k}) = u^{(k)}_{s_{k-1},s_k}(\alpha_k), \quad u^{\langle k |} \in \mathbb{R}^{r_{k-1}\times (p_k+1) r_k},$
\end{itemize}
both pointing to the same data stored in $u^{(k)}$.
\end{definition}

The latter remark about multiple pointers to the same data is a well-known conception in programming,
and will help us to  present algorithms in a more elegant way, by assigning e.g.
$u^{|k\rangle}=Q$, without a need to state explicitly $u^{(k)}_{s_{k-1},s_k}(\alpha_k) = u^{|k\rangle}(\overline{s_{k-1},\alpha_k},s_k)$ afterwards.

\begin{definition}\label{def:orth}
A TT block $u^{(k)}$ is said to be \emph{left-} resp. \emph{right-orthogonal}, if
$$
\left(u^{|k\rangle}\right)^* u^{|k\rangle} = I, \quad \mbox{or} \quad u^{\langle k|} \left(u^{\langle k|}\right)^* = I,
$$
respectively.
\end{definition}

Note that for the first and the last blocks, the left and right orthogonalities mean simply the orthogonality w.r.t. the tensor indices $\alpha_1$ and $\alpha_M$, in the same way as in the dyadic decomposition of a matrix.
Since the TT representation is not unique,
$$
u(\alpha) = \left(u^{(1)}(\alpha_1) R_1\right) \left(R_1^{-1} u^{(2)}(\alpha_2) R_2\right) R_2^{-1} \cdots R_{M-1} \left(R_{M-1}^{-1} u^{(M)}(\alpha_M)\right)
$$
for any nonsingular $R_k$ of proper sizes, the orthogonalities may be ensured without changing the whole tensor.
Indeed, neighboring TT blocks $u^{|k\rangle} u^{\langle k+1 |}$ constitute a dyadic decomposition,
which may be orthogonalized either as $q^{|k\rangle} \left(R u^{\langle k+1 |}\right)$, or $\left(u^{|k\rangle} L\right) q^{\langle k+1 |}$,
where $q^{|k\rangle}R = u^{|k\rangle}$ or $Lq^{\langle k+1 |} = u^{\langle k+1 |}$, respectively.
Repeating this procedure for all blocks yields a TT representation with the corresponding orthogonality.
Algorithm \ref{alg:TT_ort_l} explains the left orthogonalization, the right one is analogous.

\begin{algorithm}[t]
 \caption{Left TT orthogonalization} \label{alg:TT_ort_l}
 \begin{algorithmic}[1]
  \REQUIRE A tensor $u$ in the TT format.
  \ENSURE  A tensor $u$ with all TT blocks except $u^{(M)}$ left-orthogonal.
  \FOR{$k=1,\ldots,M-1$}
    \STATE Find QR decomposition $u^{|k\rangle} = q^{|k\rangle} R$, $\left(q^{|k\rangle}\right)^* q^{|k\rangle } =I$.
    \STATE Replace $u^{\langle k+1 |} = R u^{\langle k+1 |}$, \quad  $u^{|k\rangle}=q^{|k\rangle}$.
  \ENDFOR
 \end{algorithmic}
\end{algorithm}

\begin{algorithm}[t]
 \caption{TT rounding (right-to-left)} \label{alg:TT_svd_r}
 \begin{algorithmic}[1]
  \REQUIRE A tensor $u$ in the TT format, accuracies $\eps_k$, $k=1,\ldots,M-1$.
  \ENSURE  A tensor $v: \|u-v\|^2 \le \|u\|^2 \sum \eps_k^2$ with optimal TT ranks.
  \STATE Apply left orthogonalization Algorithm \ref{alg:TT_ort_l} to $u$.
  \FOR{$k=M,M-1,\ldots,2$}
    \STATE Compute SVD $u^{\langle k |} = W \diag(\sigma) V$.
    \STATE Determine minimal rank $r_{k-1}:~\sum_{s > r_{k-1}} \sigma_{s}^2 \le \eps_{k-1}^2 \|\sigma\|^2$.
    \STATE Take $r_{k-1}$ components $\tilde W_{s}=W_{s}$, $\tilde V_{s} = V_{s}$, $\tilde \sigma_{s} = \sigma_{s}$, $s = 1,\ldots,r_{k-1}$.
    \STATE Replace $u^{| k-1 \rangle} = u^{| k-1 \rangle} \cdot \tilde W \diag(\tilde \sigma)$, \quad  $v^{\langle k |}=\tilde V$.
  \ENDFOR
 \STATE $v^{|1\rangle} = u^{|1\rangle}$.
 \end{algorithmic}
\end{algorithm}

Despite the moderate complexity $\mathcal{O}(Mpr^3)$, these algorithms provide the orthogonality for large tensor train chunks, e.g. the TT interfaces.
\begin{definition}\label{def:interface}
The \emph{left-}, resp. \emph{right TT interfaces} are defined as follows:
\begin{equation*}
 \begin{split}
  U^{(1:k)}_{s_{k}}(\alpha_1,\ldots,\alpha_k) & = \sum\limits_{s_1=1}^{r_1} \cdots \sum\limits_{s_{k-1}}^{r_{k-1}} u^{(1)}_{s_1}(\alpha_1) \cdots u^{(k)}_{s_{k-1},s_k}(\alpha_k), \\
  U^{(k+1:d)}_{s_{k}}(\alpha_{k+1},\ldots,\alpha_M) & = \sum\limits_{s_{k+1}=1}^{r_{k+1}} \cdots \sum\limits_{s_{M-1}}^{r_{M-1}} u^{(k+1)}_{s_k,s_{k+1}}(\alpha_{k+1}) \cdots u^{(M)}_{s_{M-1}}(\alpha_M).
 \end{split}
\end{equation*}
\end{definition}

\begin{lemma}\label{lem:leftorth}
For a TT tensor with blocks $1,\ldots,M-1$ left-orthogonal, it holds
$$
\left(U^{(1:q)}\right)^* U^{(1:q)} = I, \quad q=1,\ldots,M-1.
$$
\end{lemma}
\begin{proof}
The product
$$
\left(\left(U^{(1:q)}\right)^* U^{(1:q)}\right)_{s_{q},s_{q}'} = \sum\limits_{\alpha_1,\ldots,\alpha_{q}} \bar U^{(1:q)}_{s_{q}}(\alpha_1,\ldots,\alpha_{q}) \cdot U^{(1:q)}_{s_{q}'}(\alpha_1,\ldots,\alpha_{q})
$$
where $\bar U$ is the complex conjugation,
may be computed by the direct summation over indices $\alpha_1,\ldots,\alpha_q$.
In the first step, we evaluate $\sum_{\alpha_1} \bar u^{(1)}_{s_1}(\alpha_1) u^{(1)}_{s_1'}(\alpha_1) = I_{s_1,s_1'}$, according to Def. \ref{def:orth}.
Let us be given $I_{s_{k-1},s_{k-1}'}$, then in the $k$-th step we compute
$$
\sum_{s_{k-1},s_{k-1}',\alpha_k} I_{s_{k-1},s_{k-1}'} \bar u^{(k)}_{s_{k-1},s_k}(\alpha_k) u^{(k)}_{s_{k-1}',s_k'}(\alpha_k) = \left(u^{|k\rangle}_{s_k}\right)^* u^{|k\rangle}_{s_k'} = I_{s_k,s_k'},
$$
that is, the induction may be continued, and finally we end up with $I_{s_q,s_q'}$.
\end{proof}

Now, it is clear how to build the TT rounding procedure: in the first step, we run Alg. \ref{alg:TT_ort_l}, making the TT format left-orthogonal, and so will be $U^{(1:M-1)}$.
Then we perform a small-sized SVD $\left(R u^{\langle M|}\right) \approx W \Sigma V$ and cast $W\Sigma$ to the block $u^{(M-1)}$.
Setting $v^{\langle M|}=V$ ensures its right orthogonality, while the left orthogonality of $U^{(1:M-2)}$ also takes place.
So we may perform the SVD of $x^{\langle d-1 |}$ and so on.
The whole procedure is summarized in Algorithm \ref{alg:TT_svd_r}.
In practice, we do not typically use different thresholds $\eps_k$ for each step,
but instead fix the global level $\eps$, and put $\eps_k=\eps/\sqrt{d-1}$ everywhere.

\subsection{Sum and product in the TT format}

Equipped with the robust reduction technique, we may think of the \emph{tensor train arithmetics},
because algebraic operations are inherited naturally from low-rank or CP decompositions.
For example, a {\bf{sum}} $w=u+v$ for $u$ and $v$ given in the form \eqref{eq:tt} casts to the TT blocks of $w$ as follows,
\begin{equation}
 w^{(M)}(\alpha_M) = \begin{bmatrix}u^{(M)}(\alpha_M) \\ v^{(M)}(\alpha_M)\end{bmatrix}, \quad w^{(k)}(\alpha_k) = \begin{bmatrix}u^{(k)}(\alpha_k) \\ & v^{(k)}(\alpha_k)\end{bmatrix},~ k=2,\ldots,M-1,
\label{eq:tt_add}
\end{equation}
and $w^{(1)}(\alpha_1) = \begin{bmatrix}u^{(1)}(\alpha_1) & v^{(1)}(\alpha_1)\end{bmatrix}$.
Matrix, pointwise and scalar products follow similarly, see below and \cite{osel-tt-2011} for more details.
The main feature of all such operations in the format is that the (TT) ranks may grow; fortunately, in many applications the rounding procedure allows to reduce them to a moderate level.

The {\bf{scalar}} (dot, inner) product $S=\dotprod{u}{v}$ of two vectors is equal to a product of all TT elements of both vectors, followed by a summation over all rank and initial indices.
However, this implementation would require $\mathcal{O}(Mpr^2(u)r^2(v)) = \mathcal{O}(Mpr^4)$ complexity.
Using the block foldings (Def. \ref{def:folded}) and auxiliary quantities, the scalar product may be computed with $\mathcal{O}(Mp(r^2(u)r(v)+r(u)r^2(v)) = \mathcal{O}(Mpr^3)$ cost, as shown in Algorithm \ref{alg:TT_dot}.

Note that Algorithm \ref{alg:TT_dot} remains consistent if the left TT ranks $r_0(u), r_0(v)$ are not equal to ones.
In this case, we simply return a matrix $S=S_1$, containing the elements $\dotprod{u_{s_0}}{v_{s_0'}}$.
It is especially convenient when we work with response surfaces for sPDEs, where $r_0$ may stand for the spatial grid size $N$.
An example is the computation of the variance in Section \ref{subsec:levelSets}.

\begin{algorithm}[t]
 \caption{Scalar product in the TT format} \label{alg:TT_dot}
 \begin{algorithmic}[1]
  \REQUIRE Tensors $u$, $v$ in the TT format.
  \ENSURE  Scalar product $S=\dotprod{u}{v}$.
  \STATE Initialize $S_{M} = v^{\langle M |} \left(u^{\langle M |}\right)^*  \in \mathbb{R}^{r_{d-1}(v) \times r_{d-1}(u)}$.
  \FOR{$k=M-1,\ldots,1$}
    \STATE $w^{|k\rangle}=v^{|k\rangle}S_{k+1} \in\mathbb{R}^{r_{k-1}(v) (p_k+1) \times r_k(u)}$.
    \STATE $S_k =  w^{\langle k |} \left(u^{\langle k |}\right)^* \in\mathbb{R}^{r_{k-1}(v) \times r_{k-1}(u)}$.
  \ENDFOR
  \RETURN $S=S_1 \in\mathbb{R}^{r_0(v) \times r_0(u)} = \mathbb{R}$.
 \end{algorithmic}
\end{algorithm}

\subsection{Block TT-Cross interpolation algorithm}
Calculation of the PCE formula \eqref{eq:k_kle_pce_coeff} could be a difficult task in tensor formats,
since the tensor indices enter the factorial functions, or even enumerate the elements of another vector, $\phi_{\alpha_1+\cdots+\alpha_M}$ in our case.
To compute the response surface in the TT format by this way, we need to be aimed with the following technique:
\begin{itemize}
 \item given a procedure to compute each element of a tensor, e.g. $\tilde\kappa_{\alpha_1,\ldots,\alpha_M}$ \eqref{eq:k_kle_pce_coeff}.
 \item build a TT approximation $\tilde\kappa_{\alpha} \approx \kappa^{(1)}(\alpha_1) \cdots \kappa^{(M)}(\alpha_M)$ using a feasible amount of elements (i.e. much less than $(p+1)^M$).
\end{itemize}
Such a tool exists, and relies on the \emph{cross interpolation} of matrices, generalized to the higher-dimensional case \cite{ot-ttcross-2010,so-dmrgi-2011proc,sav-qott-2014}.
The principal ingredient is based on the efficiency of an incomplete Gaussian elimination in approximation of a low-rank matrix, also known as the Adaptive Cross Approximation (ACA) \cite{bebe-2000,bebe-aca-2011}.
Given a matrix $U=[U(i,j)]\in\mathbb{R}^{n \times m}$, we select some few amount of ``good'' columns and rows to define the whole matrix,
\begin{equation}
U(i,j) \approx \tilde U(i,j)= \sum\limits_{s,s'=1}^r U(i,\mathbb{J}_{s}) M_{s,s'} U(\mathbb{I}_{s'},j),
\label{eq:2dcross}
\end{equation}
where $M = \left(U(\mathbb{I}, \mathbb{J})\right)^{-1}$, and $\mathbb{I} \subset \{1,\ldots,n\}$, $\mathbb{J} \subset \{1,\ldots,m\}$ are sets of indices of cardinality $r$.
It is known that there exists a quasi-optimal set of interpolating indices $\mathbb{I},\mathbb{J}$.
\begin{lemma}[Maximum volume (\emph{maxvol}) principle \cite{gt-maxvol-2001}]
 If $\mathbb{I}$ and $\mathbb{J}$ are such that $\mathrm{det} U(\mathbb{I}, \mathbb{J})$ is maximal among all $r \times r$ submatrices of $U$, then
 $$
 \|U-\tilde U\|_C \le (r+1) \min_{\mathrm{rank}(V) = r} \|U-V\|_2,
 $$
where $\| \cdot \|_C$ is the Chebyshev norm, $\|X\|_C = \max_{i,j}|X_{i,j}|$.
\label{lem:maxvol}
\end{lemma}

In practice, however, the computation of the true maxvol submatrix is infeasible, since it is a NP-hard problem.
Instead, one performs a heuristic iteration in an alternating fashion \cite{gostz-maxvol-2010}:
we start with some (e.g. random) low-rank factor $U^{(1)} \in \mathbb{R}^{n \times r}$, determine indices $\mathbb{I}$ yielding a quasi-maxvol $r \times r$ submatrix in $U^{(1)}$, and compute $U^{(2)}$ as $r$ columns of $U$ of the indices $\mathbb{I}$.
Vice versa, in the next step we find quasi-maxvol column indices in $U^{(2)}$ and calculate corresponding $nr$ elements, collecting them into the newer $U^{(1)}$, which hopefully approximates the true low-rank factor better than the initial guess.
This process continues until the convergence, which appears to be quite satisfactory in practice.

In higher dimensions we proceed similarly, thanks to the polylinear structure of the TT format,
which constitutes a recurrent matrix low-rank factorization.
Indeed, if we encapsulate together the first $k$ and last $M-k$ indices into two multiindices, the TT format \eqref{eq:tt} may be seen as the dyadic factorization with the interfaces (Def. \ref{def:interface}) playing roles of factors:
$$
u(\alpha_1,\ldots,\alpha_k;~\alpha_{k+1},\ldots,\alpha_M) = \sum\limits_{s_k=1}^{r_k} U^{(1:k)}_{s_{k}}(\alpha_1,\ldots,\alpha_k) U^{(k+1:M)}_{s_k} (\alpha_{k+1},\ldots,\alpha_M).
$$
Starting from $k=1$, we iterate the alternating cross interpolation, computing the maxvol indices in each TT block subsequently, and an update of the $k$-th block is defined by $r_{k-1} (p_k+1) r_k$ elements of the initial tensor.
Similarly to \eqref{eq:2dcross}, we compute the \emph{left} quasi-maxvol indices $\mathbb{I}^{(k)}$ from the tall matrix $U^{(1:k)}$, and the \emph{right} quasi-maxvol indices $\mathbb{J}^{(k)}$ from the wide matrix $U^{(k+1:d)}$.
Then the TT analog of \eqref{eq:2dcross} writes
\begin{equation}
u^{(k)}_{s_{k-1},s_k}(\alpha_k) = \sum_{s_{k-1}', s_k'} M^{(k-1)}_{s_{k-1},s_{k-1}'} u\left(\mathbb{I}_{s_{k-1}'}^{(k-1)}, \alpha_k, \mathbb{J}^{(k)}_{s_k'}\right) M^{(k)}_{s_{k}',s_{k}},
\label{eq:ttcross}
\end{equation}
where $M^{(k)} = \left(u\left(\mathbb{I}^{(k)}, \mathbb{J}^{(k)}\right)\right)^{-1}$, and the middle term is a sample of the sought tensor,
$u\left(\mathbb{I}_{s_{k-1}'}^{(k-1)}, \alpha_k, \mathbb{J}^{(k)}_{s_k'}\right) = u(\hat \alpha_1,\ldots,\hat\alpha_{k-1},\alpha_k,\hat\alpha_{k+1},\ldots,\hat\alpha_M)$.
This sampling runs through the whole range of $\alpha_k$ and only those of the rest indices that belong to $\mathbb{I}^{(k-1)}$ and $\mathbb{J}^{(k)}$,
i.e. $\{\hat\alpha_1,\ldots,\hat\alpha_{k-1}\}\in\mathbb{I}^{(k-1)}$ and $\{\hat\alpha_{k+1},\ldots,\hat\alpha_{M}\} \in \mathbb{J}^{(k)}$.
After that, the new $u^{(k)}$ is used in the maxvol algorithm for derivation of $\mathbb{I}^{(k)}$ as follows.
Let us be given a set $\mathbb{I}^{(k-1)}$ of $r_{k-1}$ index vectors of length $k-1$ each.
We concatenate it with all possible values of $\alpha_k$, and obtain the set $\left\{\mathbb{I}^{(k-1)}, \alpha_k\right\}$ of $r_{k-1}(p_k+1)$ vectors of length $k$ each.
The range $r_{k-1}(p_k+1)$ is exactly the row size of $u^{|k\rangle}$, and hence the maxvol will return $r_k$ values in this range.
Let us call them $\hat i_k$: the new left set $\mathbb{I}^{(k)}$ is taken as a set of tuples from $\left\{\mathbb{I}^{(k-1)}, \alpha_k\right\}$ with the labels $\hat i_k$ (among all $r_{k-1}(p_k+1)$ variants).
In the same way, the right indices $\mathbb{J}^{(k)}$ may be built in a recurrent fashion.

In total, we need only $\mathcal{O}(n_{it} M p r^2)$ entries to be evaluated, where $n_{it}$ is the number of alternating iterations (see Alg. \ref{alg:cross}), typically of the order of $10$.
In each step, given $r_{k-1}$ values $\{\hat\alpha_1,\ldots,\hat\alpha_{k-1}\}\in\mathbb{I}^{(k-1)}$, $p_k+1$ values of $\alpha_k$ and $r_k$ values $\{\hat\alpha_{k+1},\ldots,\hat\alpha_{M}\} \in \mathbb{J}^{(k)}$, it is affordable to call e.g. Formula \eqref{eq:k_kle_pce_coeff} $r_{k-1}(p_k+1)r_k$ times.

However, note that each call of \eqref{eq:k_kle_pce_coeff} throws $L$ values, corresponding to different $\ell=1,\ldots,L$.
We may account for this in various ways.
Since $\ell$ has the meaning of the reduced spatial variable, we may feature it as an additional dimension.
Nonetheless, when we will restrict the indices $\left\{\mathbb{I}^{(k-1)}, \alpha_k\right\}(\hat i_k) = \mathbb{I}^{(k)}$, we will remove some values of $\ell$ from consideration.
Therefore, a vast majority of information goes in vain: we evaluate $L$ values, but only a few of them will be
used to improve the approximation.
Another way is to run $L$ independent cross algorithms to approximate each $\kappa_{\alpha}(\ell)$ into its own TT format.
This is also not very desirable: we will have to add $L$ TT formats together, summing their ranks, which is to be followed by the TT approximation procedure \ref{alg:TT_svd_r}.
The asymptotic complexity thus scales as $\mathcal{O}(L^3)$, which was found too expensive in practice.

A better approach is to store all $\kappa_{\alpha}(\ell)$ in the same TT representation, employing the idea of the \emph{block} TT format \cite{dkos-eigb-2014}.
The resulting method has a threefold advantage: all data in each call of \eqref{eq:k_kle_pce_coeff} is assimilated, the algorithm adjusts the TT ranks automatically according to the given accuracy, and the output is returned as a single optimally-compressed TT format, convenient for further processing.

\begin{algorithm}[t]
 \caption{Block cross approximation of a tensor in the TT format} \label{alg:cross}
 \begin{algorithmic}[1]
  \REQUIRE A function to evaluate $u_\ell(\alpha_1,\ldots,\alpha_M)$, initial TT guess $u^{(1)}(\alpha_1)\cdots u^{(M)}(\alpha_M)$, relative accuracy threshold $\eps$.
  \ENSURE  Improved TT approximation $u^{(1)}(\ell,\alpha_1) u^{(2)}(\alpha_2) \cdots u^{(M)}(\alpha_M)$.
  \STATE Initialize $\mathbb{I}^{(0)}=[]$, \quad $U^{(1:0)}(\mathbb{I}^{(0)})=1$, \quad $\mathbb{J}^{(M)} = []$, \quad $U^{(M+1:M)}(\mathbb{J}^{(M)}) = 1$.
\FOR{$\mathrm{iteration}=1,2,\ldots,n_{it}$ or until convergence}
  \FOR[Forward sweep]{$k=1,2,\ldots,M-1$}
    \IF[All indices are available, assimilate the information]{$\mathrm{iteration}>1$}
      \STATE Evaluate the tensor at cross indices $v^{(k)}_\ell(\alpha_k) = u_\ell \left(\mathbb{I}^{(k-1)}, \alpha_k, \mathbb{J}^{(k)}\right) \in \mathbb{C}^{r_{k-1} \times r_k}$.
      \STATE Invert the matrices $M^{(k-1)} = \left(U^{(1:k-1)}(\mathbb{I}^{(k-1)})\right)^{-1}$, \quad $M^{(k)} = \left(U^{(k+1:M)}(\mathbb{J}^{(k)})\right)^{-1}$.
      \STATE Compute the common block $\hat u^{(k)}(\alpha_k,\ell) = M^{(k-1)} v^{(k)}_\ell(\alpha_k) M^{(k)}$.
      \STATE Compute truncated SVD $\hat u^{(k)}(\alpha_k,\ell) \approx u^{(k)}(\alpha_k) \diag(\sigma) V(\ell)$, ensure accuracy $\eps$.
    \ELSE[Warmup sweep: the indices are yet to be built]
      \STATE Find QR decomposition $u^{|k\rangle} = q^{|k\rangle} R$, $\left(q^{|k\rangle}\right)^* q^{|k\rangle } =I$.
      \STATE Replace $u^{\langle k+1 |} = R u^{\langle k+1 |}$, \quad  $u^{|k\rangle}=q^{|k\rangle}$.
    \ENDIF
    \STATE Compute $V^{\langle k |} = U^{(1:k-1)}(\mathbb{I}^{(k-1)}) u^{\langle k |}$.
    \STATE Find \emph{local} maxvol indices $\hat i_k = \mathtt{maxvol}\left(V^{|k\rangle}\right) \subset \{1,\ldots,r_{k-1}(p_k+1)\}$.
    \STATE Sample global indices $\mathbb{I}^{(k)} = \left\{\mathbb{I}^{(k-1)}, \alpha_k\right\}(\hat i_k)$, \quad $U^{(1:k)}(\mathbb{I}^{(k)}) = V^{|k\rangle}(\hat i_k) \in \mathbb{C}^{r_{k} \times r_{k}}$.
  \ENDFOR
  \FOR[Backward sweep]{$k=M,M-1,\ldots,2$}
    \STATE Evaluate the tensor at cross indices $v^{(k)}_\ell(\alpha_k) = u_\ell \left(\mathbb{I}^{(k-1)}, \alpha_k, \mathbb{J}^{(k)}\right) \in \mathbb{C}^{r_{k-1} \times r_k}$.
    \STATE Invert the matrices $M^{(k-1)} = \left(U^{(1:k-1)}(\mathbb{I}^{(k-1)})\right)^{-1}$, \quad $M^{(k)} = \left(U^{(k+1:M)}(\mathbb{J}^{(k)})\right)^{-1}$.
    \STATE Compute the common block $\hat u^{(k)}(\ell,\alpha_k) = M^{(k-1)} v^{(k)}_\ell(\alpha_k) M^{(k)}$.
    \STATE Compute truncated SVD $\hat u^{(k)}(\ell,\alpha_k) \approx V(\ell) \diag(\sigma) u^{(k)}(\alpha_k)$, ensure accuracy $\eps$.
    \STATE Compute $V^{|k\rangle} = u^{| k \rangle}  U^{(k+1:M)}(\mathbb{J}^{(k)})$
    \STATE Find \emph{local} maxvol indices $\hat j_k = \mathtt{maxvol}\left(V^{\langle k |}\right) \subset \{1,\ldots,(p_k+1) r_k\}$.
    \STATE Global indices $\mathbb{J}^{(k-1)} = \left\{\alpha_k, \mathbb{J}^{(k)}\right\}(\hat j_k)$, \quad $U^{(k:M)}(\mathbb{J}^{(k-1)}) = V^{\langle k |}(\hat j_k) \in \mathbb{C}^{r_{k-1} \times r_{k-1}}$.
  \ENDFOR
  \STATE Evaluate the first block $u^{(1)}(\ell,\alpha_1) = u_\ell \left(\alpha_1, \mathbb{J}^{(2)} \right) \left(U^{(2:M)}(\mathbb{J}^{(2)})\right)^{-1}$.
\ENDFOR
 \end{algorithmic}
\end{algorithm}

For the uniformity of the explanation we assume that we have a procedure that, given an index $\alpha_1,\ldots,\alpha_M$, throws $L$ values $u_\ell(\alpha)$, $\ell=1,\ldots,L$.
In other words, we will call $u_\ell(\alpha):=\tilde\kappa_{\alpha}(\ell)$ until the end of this section.
When the block $u_{\ell}(\mathbb{I}^{(k-1)},\alpha_k,\mathbb{J}^{(k)})$ is evaluated, we modify \eqref{eq:ttcross} as follows:
\begin{equation}
\hat u^{(k)}_{s_{k-1},s_k}(\alpha_k,\ell) = \sum_{s_{k-1}', s_k'} M^{(k-1)}_{s_{k-1},s_{k-1}'} u_{\ell}\left(\mathbb{I}^{(k-1)}_{s_{k-1}'},\alpha_k,\mathbb{J}^{(k)}_{s_k'}\right) M^{(k)}_{s_{k}',s_{k}}.
\label{eq:ttb_cross}
\end{equation}
We are now looking for the basis in $\alpha$ that is best suitable for all $u_{\ell}$.
Hence, we compute the (truncated) singular value decomposition (cf. Alg. \ref{alg:TT_svd_r}):
\begin{equation}
\hat u^{(k)}_{s_{k-1},s_k}(\alpha_k,\ell) \approx \sum_{s_k'=1}^{r_k'} u^{(k)}_{s_{k-1},s_k'}(\alpha_k)  \sigma_{s_k'} V_{s_k',s_k}(\ell).
\label{eq:btt_svd}
\end{equation}
The new rank $r_k'$ is usually selected via the Frobenius-norm error criterion: we choose a minimal $r_k'$ such that $\|\hat u^{(k)} -  u^{(k)} \mathrm{diag}(\sigma) V\| \le \eps \|\hat u^{(k)}\|$.
Nothing else than the left singular vectors are collected into the updated TT block $u^{(k)}$, and the TT rank is updated, $r_k:=r_k'$.
After that, we proceed to the next block $k+1$, and so on.

Analogous scheme is written for the backward iteration, i.e. if we go to the block $k-1$ in the next step.
We usually go back and forth through the tensor train in such an alternating fashion until the TT representation stabilizes, or the maximal number of iterations is hit.
We summarize the whole procedure in the \emph{Block TT-Cross} Algorithm \ref{alg:cross}. \\

\section{Statistical calculations in the TT format}

\subsection{Computation of the PCE in the TT format via the cross interpolation}\label{sec:tt-kle-pce}
Equipped with Alg. \ref{alg:cross}, we may apply it for the PCE approximation, passing Formula \eqref{eq:k_kle_pce_coeff} as a function $u_\ell(\alpha)$ that evaluates tensor values on demand.
The initial guess may be even a rank-1 TT tensor with all blocks populated by random numbers, since the cross iterations will adapt both the representation and TT ranks.

Formula \eqref{eq:ttb_cross} together with \eqref{eq:k_kle_pce_coeff} requires $\mathcal{O}(r^2 p (MN+NL) + r^3 pL)$ operations,
and the complexity of the SVD \eqref{eq:btt_svd} is $\mathcal{O}(r^3 p L \cdot \min\{p,L\})$,
so it is unclear in general which term will dominate.
For large $N$, we are typically expecting that it is the evaluation \eqref{eq:ttb_cross}.
However, if $N$ is moderate (below $1000$), but the rank is large ($\sim 100$), the singular value decomposition may win the race.
For the whole algorithm, assuming also $L \sim M$, we can thus claim the $\mathcal{O}(n_{it} M^2 N p r^3)$ complexity,
which is lower than $\mathcal{O}(M p L^3)$ that we could had if we run independent cross algorithms for each $\ell$.

As soon as the reduced PCE coefficients $\tilde\kappa_{\alpha}(\ell)$ are computed, the initial expansion \eqref{eq:k_kle_pce} comes easily.
Indeed, after the backward cross iteration, the $\ell$ index lives in the first TT block, and we may let it play the role of a ``zeroth'' TT rank index,
\begin{equation}
\tilde\kappa_{\alpha}(\ell) = \sum_{s_1,\ldots,s_{M-1}} \kappa^{(1)}_{\ell,s_1}(\alpha_1) \kappa^{(2)}_{s_1,s_2}(\alpha_2) \cdots \kappa^{(M)}_{s_{M-1}}(\alpha_M).
\end{equation}
For $\ell=0$ we extend this formula such that $\tilde\kappa_{\alpha}(0)$ is the first identity vector $\delta_0$, cf. \eqref{eq:f}.
Now, collect the spatial components into the ``zeroth'' TT block,
\begin{equation}
\kappa^{(0)}(x) = \begin{bmatrix} \kappa^{(0)}_\ell(x) \end{bmatrix}_{\ell=0}^{L} = \begin{bmatrix}\bar\kappa(x) & v_1(x) & \cdots v_L(x) \end{bmatrix},
\end{equation}
then the PCE \eqref{eq:pce_k} writes as the following TT format,
\begin{equation}
\kappa_{\alpha}(x) = \sum_{\ell,s_1,\ldots,s_{M-1}} \kappa^{(0)}_{\ell}(x) \kappa^{(1)}_{\ell,s_1}(\alpha_1) \cdots \kappa^{(M)}_{s_{M-1}}(\alpha_M).
\label{eq:k_pce_tt}
\end{equation}

\subsection{Stiffness Galerkin operator in the TT format}\label{sec:stiff_tt}

With the sparse set $\J_{M,p}^{sp}$, we compute the operator directly by \eqref{eq:K_gen}, evaluating all elements of $\BDelta$ explicitly \cite{ZanderDiss, sglib}.
This is a very common way in stochastic Galerkin method.
It takes therefore $\mathcal{O}((\#\J_{M,p}^{sp})^3)$ operations and storage, and can be a bottleneck in the solution scheme, see the numerical experiments below.

In the TT format with the full set $\J_{M,p}$ we may benefit from the rank-1 separability of $\BDelta$ \eqref{eq:Delta}, see also \cite{wahnert-stochgalerkin-2014}.
Indeed, each index $\alpha_m,\beta_m,\nu_m$ varies in the range $\{0,\ldots,p\}$ independently on the others.
Given the PCE \eqref{eq:k_pce_tt} in the TT format, we split the whole sum over $\nu$ in \eqref{eq:K_gen} to the individual variables,
\begin{equation}
\sum\limits_{\nu \in \J_{M,p}} \BDelta_{\nu} \tilde \kappa_{\nu}(\ell) = \sum_{s_1,\ldots,s_{M-1}}  \left(\sum_{\nu_1=0}^{p} \Delta_{\nu_1} \kappa^{(1)}_{\ell,s_1}(\nu_1)\right) \otimes \cdots \otimes \left(\sum_{\nu_M=0}^{p} \Delta_{\nu_M} \kappa^{(M)}_{s_{M-1}}(\nu_M) \right).
\label{eq:deltakappa_rank1}
\end{equation}
Similar reduction of a large summation to one-dimensional operations arises also in lattice structured calculations of interaction potentials \cite{vekh-lattice-2014}.
Introduce the spatial ``zeroth'' TT block for the operator, agglomerating \eqref{eq:K_l},
\begin{equation*}
\K^{(0)}(i,j) = \begin{bmatrix} \K^{(0)}_\ell(i,j) \end{bmatrix}_{\ell=0}^{L} = \begin{bmatrix}K_0(i,j) & K_1(i,j) & \cdots & K_L(i,j) \end{bmatrix}, \quad i,j=1,\ldots,N,
\end{equation*}
and denote the parametric blocks in \eqref{eq:deltakappa_rank1} as
$\K^{(m)}_{s_{m-1},s_m} = \sum_{\nu_m=0}^{p} \Delta_{\nu_m} \kappa^{(m)}_{s_{m-1},s_m}(\nu_m)$ for $m=1,\ldots,M$,
then the TT representation for the operator writes
\begin{equation}
\K = \sum_{\ell,s_1,\ldots,s_{M-1}} \K^{(0)}_{\ell} \otimes \K^{(1)}_{\ell,s_1} \otimes \cdots \otimes \K^{(M)}_{s_{M-1}} \in \mathbb{R}^{(N \cdot \#\J_{M,p}) \times (N \cdot \#\J_{M,p})},
\label{eq:K_tt}
\end{equation}
where the TT ranks coincide with those of $\tilde \kappa$ (this is important).

One interesting property of the Hermite triples is that $\Delta_{\alpha,\beta,\nu}=0$ if e.g. $\nu>\alpha+\beta$.
That is, if we set the same $p$ for $\alpha$, $\beta$ and $\nu$, in the assembly of \eqref{eq:K_gen} we may miss some components, corresponding to $\alpha>p/2$, $\beta>p/2$.
Therefore, it could be reasonable to vary $\nu$ in the range $\{0,\ldots,2p\}$, and hence assemble $\tilde \kappa$ in the set $\J_{M,2p}$.
However, in the sparse set it would inflate the storage of $\BDelta$ and $\K$ significantly.
Contrarily, in the TT format it is not a harm: first, the TT ranks barely depend on $p$, and hence the storage scales linearly; second, we may not care about the sparsity of the matrix handled in the TT format.
So, we may assure ourselves with the exact calculation of the stiffness operator.


\subsection{Computation of level sets, frequency, mean and maximal values, and variance in the TT format}\label{subsec:levelSets}
In this section we discuss how to calculate standard statistical outputs from the response surface, staying in the TT format.

The very first example is the transition from the Galerkin indices $\alpha$ (e.g. expansion \eqref{eq:pce_k_coeffs}) to the stochastic coordinates $\theta$, according to \eqref{eq:pce_k}.
Since $\J_{M,p}$ is a tensor product set, Formula \eqref{eq:pce_k} can be evaluated in the TT format with no change of the TT ranks, similarly to the construction of the stiffness matrix in the previous subsection,
\begin{equation}
u(x,\theta) = \sum_{s_0,\ldots,s_{M-1}} u^{(0)}_{s_0}(x) \left(\sum_{\alpha_1=0}^{p} h_{\alpha_1}(\theta_1) u^{(1)}_{s_0,s_1}(\alpha_1)\right) \cdots \left(\sum_{\alpha_M=0}^{p} h_{\alpha_M}(\theta_M) u^{(M)}_{s_{M-1}}(\alpha_M)\right).
\label{eq:alpha_to_theta}
\end{equation}

Other examples, which are important for the computation of cumulative distributions, are the characteristic, frequency and level set functions \cite{litvinenko-spde-2013}.
\begin{definition}[Characteristic, Level Set, Frequency]\label{defn:LevelSetFrequency}
Let $\mathbb{I} \subset \mathbb{R}$ be a subset of real numbers.
\begin{itemize}
\item The \emph{characteristic} of $u$ at $\mathbb{I}$ is defined pointwise for all $\thetab = (\theta_1,\ldots,\theta_M) \in \mathbb{R}^{M}$ as follows,
\begin{equation}\label{equ:chi}
    \chi_{\mathbb{I}}(\theta):=\left\{
                         \begin{array}{ll}
                           1, & u(\thetab) \in \mathbb{I}, \\
                           0, & u(\thetab) \nin \mathbb{I}.
                         \end{array}
                       \right.
\end{equation}
\item The \emph{level set} reads $\mathcal{L}_{\mathbb{I}}(\thetab):= u(\thetab) \chi_{\mathbb{I}}(\thetab)$.
\item The \emph{frequency} is a number of $u(\theta)$ values, falling into $\mathbb{I}$, i.e.
$$\mathcal{F}_{\mathbb{I}}:= \# \supp \chi_{\mathbb{I}} = \sum_{\theta_1}\cdots\sum_{\theta_M} \chi_{\mathbb{I}}(\theta_1,\ldots,\theta_M).$$
\end{itemize}
Note that the frequency is defined only if the set of admissible values for $\theta$ is finite, while other functions allow continuous spaces.
\end{definition}

The \emph{mean} value of $u$, in the same way as in $\kappa$, can be derived as the PCE coefficient at $\alpha=(0,\ldots,0)$, $\bar{u}(x) = u_{0}(x)$.
The \emph{covariance} is more complicated and requires both multiplication (squaring) and summation over $\alpha$.
By definition, the covariance reads
\begin{equation*}
\begin{split}
\cov_u(x,y) & = \int\limits_{\mathbb{R}^M} \left(u(x,\thetab) - \bar u(x)\right) \left(u(y,\thetab) - \bar u(y)\right) d\mu(\thetab) \\
 & = \sum\limits_{\substack{\alpha,\beta \neq (0,\ldots,0), \\ \alpha, \beta \in \mathcal{J}_{M,p}}} u_{\alpha}(x) u_{\beta}(y) \int H_{\alpha}(\thetab) H_{\beta}(\thetab) d\mu(\thetab).
\end{split}
\end{equation*}
Knowing that $\int H_{\alpha}(\thetab) H_{\beta}(\thetab) d\mu(\thetab) = \alpha! \delta_{\alpha,\beta}$,
we take the parametric chunk of the TT decomposition of $u_{\alpha}$, multiply it with the Hermite mass matrix (rank-preserving),
\begin{equation}
v_{\alpha}(\ell) :=  u_{\alpha}(\ell) \sqrt{\alpha!} =  \sum_{s_1,\ldots,s_{M-1}} \left(u^{(1)}_{\ell,s_1}(\alpha_1) \sqrt{\alpha_1!} \right) \cdots  \left(u^{(M)}_{s_{M-1}}(\alpha_M) \sqrt{\alpha_M!} \right),
\label{eq:param_chunk_masscorr}
\end{equation}
and then take the scalar product $C = \left[C_{\ell,\ell'}\right]$, where $C_{\ell,\ell'} = \dotprod{v(\ell)}{v(\ell')}$ with $v$ defined in \eqref{eq:param_chunk_masscorr}, using Algorithm \ref{alg:TT_dot}.
Given the TT rank bound $r$ for $u_{\alpha}(\ell)$, we deduce the $\mathcal{O}(Mpr^3)$ complexity of this step\footnote{Note that this estimate does not contain $L$. Since $u_{\alpha}(\ell)$ is represented in the common TT format for all $\ell$, the KLE order $L$ plays the role of the TT rank, and is already accounted for in $r$, e.g. $L \le r$}.
After that, the covariance is given by the product of $C$ with the spatial TT blocks,
\begin{equation}
\cov_u(x,y) = \sum\limits_{\ell,\ell'=0}^{L} u^{(0)}_\ell(x) C_{\ell,\ell'} u^{(0)}_\ell(y),
\label{eq:cov_tt}
\end{equation}
where $u^{(0)}_{\ell}$ is the ``zeroth'' (spatial) TT block of the decomposition \eqref{eq:alpha_to_theta}.
Given $N$ degrees of freedom for $x$, the complexity of this step is $\mathcal{O}(N^2 L^2)$.
Note that a low-rank tensor approximation of a large covariance matrix is very important in e.g. Kriging  \cite{LitvNowak13}.
The \emph{variance} is nothing else than the diagonal of the covariance, $\var_u(x) = \cov_u(x,x)$.

The functions in Def. \ref{defn:LevelSetFrequency} are more difficult to compute.
Given the characteristic, the level set is yielded by the Hadamard product in the TT format, and the frequency is the scalar product with the all-ones tensor, $\mathcal{F}_{\mathbb{I}} = \dotprod{\chi_{\mathbb{I}}}{\mathtt{1}}$.
However, the characteristic is likely to develop large TT ranks, unless $\mathbb{I}$ is such that $\chi_{\mathbb{I}}(\theta)=1$ selects a hypercube of $\theta$, aligned to the coordinate axes of $\mathbb{R}^M$.
It can be computed using either the cross Algorithm \ref{alg:cross}, which takes Formula \eqref{equ:chi} as the function that evaluates a high-dimensional array $\chi$ at the index $\thetab$, or the Newton method for the sign function.
In both cases we may face a rapid growth of TT ranks during the cross or Newton iterations.

The approximate \emph{maximum} of the response surface can be obtained during the cross approximation Algorithm \ref{alg:cross}: it was found that the maxvol indices are good candidates for the global maximum in a low-rank matrix \cite{gostz-maxvol-2010}, so we may look for the extremal element among  $u\left(\mathbb{I}^{(k-1)}, \alpha_k, \mathbb{J}^{(k)}\right)$ only.

\subsection{Sensitivity analysis in the TT format}\label{sec:SA}
Another interesting counterpart of the covariance is the \emph{Sobol sensitivity index} \cite{crestaux-sapce-2009}, the ``partial'' variance,
\begin{equation}
S_q = \dfrac{D_q}{D}, \qquad D_q = \int\limits_{\mathbb{R}^{|q|}} u_q^2(x,\thetab_q) d\mu(\thetab_q), \quad D = D_{1,\ldots,M},
\label{eq:sobol}
\end{equation}
where $q \subset \{1,\ldots,M\}$ is a set of stochastic variables of interest, the sub-vector of $\thetab$ is selected as $\thetab_q = (\theta_{q_1},\ldots,\theta_{q_{|q|}})$ and the solution is restricted to $q$ as follows,
\begin{equation}
u_q(x,\thetab_q) = \int\limits_{\mathbb{R}^{M-|q|}} u(x,\thetab) d\mu(\thetab_{\perp q}) - \sum\limits_{t \subset q,~t \neq q} u_t(x,\thetab_t),
\label{eq:u_q}
\end{equation}
where $\thetab_{\perp q}$ is the complement to $\thetab_q$, i.e. $\thetab$ with $q$ \emph{excluded}.
The sensitivity indices can be nicely computed in the TT format, similarly to the covariance.
First, supposing without loss of generality that $\theta_{1},\ldots,\theta_{\tau}$ are contained in $\thetab_{\perp q}$, the integration w.r.t. $d\mu(\thetab_{\perp q})$ is performed simply by extraction of the zeroth-order polynomial coefficients in the corresponding variables,
\begin{equation*}
\hat u_{q}(l,\alpha_{q}) = \sum_{s_1,\ldots,s_{M-1}} \left(u^{(1)}_{l,s_1}(0) \cdots u^{(\tau)}_{s_{\tau-1},s_{\tau}}(0) \right) u^{(\tau+1)}_{s_{\tau},s_{\tau+1}}(\alpha_{\tau+1}) \cdots u^{(M)}_{s_{M-1}}(\alpha_M),
\end{equation*}
which can be seen as a $|q|$-dimensional TT format.
Proceeding in the same way as in \eqref{eq:alpha_to_theta}, we may cast $\hat u_{q}(l,\alpha_{q}) \rightarrow \hat u_{q}(x,\thetab_{q})$, which is exactly the integral in \eqref{eq:u_q}.
However, the Sobol formula \eqref{eq:sobol} does not actually require it.
We evaluate
\begin{equation}
u_q(l,\alpha_q) = \hat u_{q}(l,\alpha_q) - \sum\limits_{t \subset q,~t \neq q} \hat u_t(l,\alpha_t),
\label{eq:u_q_l}
\end{equation}
and as in the previous subsection, multiply $u_{q}$ with the mass matrix like in \eqref{eq:param_chunk_masscorr} and obtain $v_{q}(l,\alpha_q)$.
Computing the TT-scalar product $C_q=\dotprod{v_q}{v_q}$, and multiplying it with the spatial blocks as in \eqref{eq:cov_tt}, we obtain $D_q$ \eqref{eq:sobol}.

On the other hand, we may think of using rough approximate sensitivities to estimate the ``importance'' of stochastic variables, and hence their optimal order in the TT format.
A proper ordering of indices in a tensor product format can reduce the ranks and storage drastically.


\section{Numerical Experiments}
\label{sec:numerics}
We verify the approach on the elliptic stochastic equation \eqref{eq:elliptic} with $f = f(x) = 1$.
We assume that the permeability coefficient $\kappa(x,\omega)$ obeys the $\beta\{5,2\}$-distribution, shifted by $1$ to ensure the non-negativity,
$$
\kappa(x,\omega) = 1+\gamma(x,\omega), \qquad \gamma(x,\omega)\in[0,1] \quad \mbox{obeys} \quad P(\gamma)=\frac{1}{B(5,2)} \gamma^4 (1-\gamma)^1,
$$
and the covariance matrix $\cov_{\kappa}(x,y) = \exp\left(-(x-y)^2/\sigma^2\right)$ with $\sigma=0.3$.
This Gaussian covariance matrix yields formally an exponential decay of the KLE coefficients \cite{SCHWAB2006,Bieri2009, Hcovariance},
but the actual rate is rather slow, due to the small correlation length $\sigma$.
The spatial domain $D$ is the two-dimensional $L$-shape area, cut from the square $[-1,1]^2$, see Fig. \ref{fig:u_space}.
To generate the spatial mesh, we use the standard PDE Toolbox in MATLAB with one level of refinement,
which yields $557$ degrees of freedom in total, and $477$ inner points.

All utilities related to the Hermite PCE are taken from the \emph{sglib} \cite{sglib}, including discretization and solution routines in the sparse polynomial set $\J_{M,p}^{sp}$.
However, to work with the TT format (for full $\J_{M,p}$), we employ the \emph{TT-Toolbox} \cite{tt-toolbox}.
We use the modules of \emph{sglib} for low-dimensional stages (e.g. KLE of the covariance matrix),
and replace the parts corresponding to high-dimensional calculations by the TT algorithms.
The two most important instances are the following: \texttt{amen\_cross.m} from the TT-Toolbox, the block cross interpolation Algorithm \ref{alg:cross} for the TT approximation of $\tilde \kappa_{\alpha}$ \eqref{eq:k_kle_pce_coeff},
and \texttt{amen\_solve.m} from the companion package tAMEn \cite{tamen}, the linear system solver in the TT format (see \cite{ds-amr1-2013,ds-amr2-2013,d-tamen-2014}).
Computations were conducted in MATLAB R2013b on a single core of the Intel Xeon E5-2670 CPU at 2.60GHz, provided by the Max Planck Institute, Leipzig.

The solution scheme consists of three stages: generation of the random permeability coefficient $\kappa(x,\omega)$ (the right hand side and the Dirichlet boundary conditions are assumed deterministic), computing of the stochastic stiffness (Galerkin) matrix, and the solution of the linear system.
We first report the performances of each stage separately, since they are significantly different,
and in the end we show the agglomerated results of the whole solution process.

\subsection{Computation of the PCE for the permeability coefficient}
To evaluate the coefficient according to \eqref{eq:k_kle_pce_coeff}, \eqref{eq:k_kle_pce} directly, which is feasible in the sparse set $\J_{M,p}^{sp}$,
we use the \emph{sglib} routine \texttt{expand\_field\_kl\_pce}.
For the TT calculations with $\J_{M,2p}$, we prepare the KLE components, and run the \texttt{amen\_cross} (see Algorithm \ref{alg:cross}) with the TT approximation threshold $\eps=10^{-4}$.

The computational times are shown in Table \ref{tab:k_time}.
The most time consuming step is the cross algorithm, which is $10$--$50$ times slower, compared to the sparse evaluation.
This is due to the relatively high complexity of the SVD operations in the former, which face the TT ranks up to a hundred.
However, notice that the complexity of the TT method scales linearly with $p$, while in the sparse set it may increase drastically.

Another issue is the rapid growth of the computational time with the stochastic dimension.
This phenomenon reflects the slow decay of the KL eigenvalues: since $\sqrt{\mu_M}/\sqrt{\mu_1}$ in \eqref{eq:k_kle} is larger than $\eps$, all components $v_l$ and $g_m$ are treated as significant, and hence the TT ranks grow linearly (or even a bit stronger) with the dimension.
For example, $r\sim 70$ for $M=10$, but $r \sim 200$ already for $M=30$.
This is an intrinsic property of the problem with the given distribution parameters, in particularly the correlation length $\sigma=0.3$ in the covariance matrix.
Typically, one may expect faster decay rates for larger correlation lengths, and hence faster computations.

\begin{table}[t]
\caption{CPU times (sec.) of the permeability assembly}
\label{tab:k_time}
\begin{tabular}{c|ccc|ccc}
    & \multicolumn{3}{c|}{Sparse} & \multicolumn{3}{c}{TT} \\ \hline
$p$ $\backslash$ $M$ & 10     & 20      & 30      & 10      & 20     & 30 \\ \hline
1                    & 0.2924 & 0.3113  & 0.3361  &  3.6425 & 68.505 & 616.97  \\
2                    & 0.3048 & 0.3556  & 0.4290  &  6.3861 & 138.31 & 1372.9  \\
3                    & 0.3300 & 0.5408  & 1.0302  &  8.8109 & 228.92 & 2422.9  \\
4                    & 0.4471 & 1.7941  & 6.4483  &  10.985 & 321.93 & 3533.4  \\
5                    & 1.1291 & 7.6827  & 46.682  &  14.077 & 429.99 & 4936.8  \\
\end{tabular}
\end{table}

\begin{table}[t]
\caption{Discrepancies in the permeability coefficients at $\J_{M,p}^{sp}$}
\label{tab:k_err}
\begin{tabular}{c|ccccc}
$p$    &  1       & 2        & 3        & 4        & 5 \\ \hline
$M=10$ &  2.21e-4 &  3.28e-5 &  1.22e-5 &  4.15e-5 &  6.38e-5 \\
$M=20$ &  3.39e-4 &  5.19e-5 &  2.20e-5 &  ---     &  --- \\
$M=30$ &  5.23e-2 &  5.34e-2 &  ---     &  ---     &  --- \\
\end{tabular}
\end{table}

Correctness may be verified by comparing those TT coefficients presenting in the sparse multi-index set $\J_{M,p}^{sp}$ with the \emph{sglib} output.
Note that for high $M$ and $p$ the cardinality of $\J_{M,p}^{sp}$ becomes too large, and hence such tests were not conducted.
The average errors are shown in Table \ref{tab:k_err}, and we may observe that the TT procedure delivers accurate coefficients, at least in the most significant components.

The number of evaluations by Formula \eqref{eq:k_kle_pce_coeff} is proportional to $r^2$, and is of the order of $4 \cdot 10^5$ for $M=10$, and $6 \cdot 10^7$ for $M=30$ (and $p=3$), which is nevertheless much smaller than the total amount, e.g. $\# \J_{30,6} \simeq 10^{25}$.
The overhead constant, i.e. the number of evaluations divided by the ultimate number of TT elements, is of the order of $10$ in all tests.

\subsection{Generation of the stiffness elliptic operator}
The Galerkin stiffness operator \eqref{eq:K_gen} was assembled from the PCE of $\kappa$ as described in Section \ref{sec:stiff_tt}.
In the \emph{sglib} computations with $\J_{M,p}^{sp}$, we employ the procedure \texttt{kl\_pce\_compute\_operator\_fast}.
The TT approach does not require a dedicated procedure, since it is a rank-preserving operation.

The computational times are shown in Table \ref{tab:oper_time}.
Here the situation is completely opposite to the previous experiment:
the complexity in the TT format is negligibly small, and stable with respect to both $p$ and $M$,
while in the sparse representation the direct handling of $(\#\J_{M,p}^{sp})^3$ entries
quickly makes the calculations infeasible.
If we need to solve a stochastic equation, not just store a given field $\kappa$,
the TT approach may become preferable even despite the slower PCE and solution procedures,
see Fig. \ref{fig:err_u_var}.
Excessive memory and time demands of the operator assembly stage alone may make the whole sparse technique an outsider.

\begin{table}[t]
\caption{CPU times (sec.) of the operator assembly}
\label{tab:oper_time}
\begin{tabular}{c|ccc|ccc}
    & \multicolumn{3}{c|}{Sparse} & \multicolumn{3}{c}{TT} \\ \hline
$p$ $\backslash$ $M$ & 10     & 20      & 30           & 10     & 20       & 30 \\ \hline
1                    & 0.1226 & 0.2171  & 0.3042       & 0.1124 & 0.2147   & 0.3836 \\
2                    & 0.1485 & 2.1737  & 26.510       & 0.1116 & 0.2284   & 0.5438 \\
3                    & 2.2483 & 735.15  & ---          & 0.1226 & 0.2729   & 0.8403 \\
4                    & 82.402 & ---     & ---          & 0.1277 & 0.2826   & 1.0832 \\
5                    & 3444.6 & ---     & ---          & 0.2002 & 0.3495   & 1.1834 \\
\end{tabular}
\end{table}

\subsection{Solution of the stochastic PDE}
The last stage is the actual solution of the linear system $\K\u = \f$, with
$\K$ computed by \eqref{eq:K_gen} and \eqref{eq:K_tt}, and $\f$ is according to \eqref{eq:f}.
In the sparse \emph{sglib} format, we use the PCG method,
with the deterministic stiffness matrix as a preconditioner, i.e. $P = (K_0)^{-1} \otimes I$,
where $I$ is the identity of the same size as $\#\J^{sp}_{M,p}$.
In the TT approach, the preconditioner is of the same form, but the identity is of size $\#\J_{M,p}$, and is kept in the format.
As the TT solution algorithm, we employ the \texttt{amen\_solve} procedure.

The CPU times of the solution stage are shown in Table \ref{tab:u_time}.
In both sparse and TT schemes, the solution time grows rapidly with both $p$ and $M$, which is subject not only to the cardinalities of the index sets and TT ranks, but also the condition numbers of $\K$, which increase with $p$ and $M$ as well.
Nevertheless, for large $p$ the TT scheme may overcome the sparse approach.
Recall the previous subsection: for extreme cases, the matrix $\K$ just cannot be assembled on the sparse set with reasonable time and memory costs.
Though the TT method may seem slow, in such situations it appears to be more reliable.

We also compare the errors in the outputs with respect to the reference solutions in both representations.
Typically, we need only a few statistical quantities of the solution.
Since the mean value is a particular Galerkin coefficient in our setting, it is resolved up to the accuracy $10^{-6}$ in all tests.
So, we track the behavior of the second moment, i.e. the covariance of the solution.
The reference covariance matrix $\cov_u^{\star} \in \mathbb{R}^{N \times N}$ is computed in the TT format with $p=5$, and the discrepancies in the results with smaller $p$ are calculated in average over all spatial points,
$$
|\cov_u - \cov_u^\star| = \frac{\sqrt{\sum_{i,j}(\cov_u - \cov_u^\star)_{i,j}^2}}{\sqrt{\sum_{i,j}(\cov_u^\star)_{i,j}^2}}.
$$
The results are shown in Table \ref{tab:err_u_var} and Fig. \ref{fig:err_u_var}.
We see that the error in the TT computations is smaller compared to the sparse approach, since the full set $\J_{M,p}$ is wider than the sparse one.
When we choose $p\ge 4$, the error stabilizes at the level of tensor approximation $\eps=10^{-4}$.
Taking into account the computational cost in the right plane of Fig. \ref{fig:err_u_var},
we may conclude that the TT methods become preferable for large $p$, where they deliver the same accuracy for lower complexity.
For low $p$, the sparse ansatz is faster, since it does not involve expensive singular value decompositions.
The mean and variance of the solution are shown in Fig. \ref{fig:u_space}.

\begin{table}[t]
\caption{CPU times (sec.) of the solution}
\label{tab:u_time}
\begin{tabular}{c|ccc|ccc}
    & \multicolumn{3}{c|}{Sparse} & \multicolumn{3}{c}{TT} \\ \hline
$p$ $\backslash$ $M$ & 10     & 20      & 30           & 10      & 20        & 30 \\ \hline
1                    & 0.2291 & 1.169   & 0.4778       & 1.074   & 9.3492    & 51.177  \\
2                    & 0.3088 & 2.123   & 3.2153       & 1.681   & 27.014    & 173.21  \\
3                    & 0.8112 & 14.04   & ---          & 2.731   & 56.041    & 391.59  \\
4                    & 5.7854 & ---     & ---          & 7.237   & 142.87    & 1497.1  \\
5                    & 61.596 & ---     & ---          & 45.51   & 866.07    & 5362.8  \\
\end{tabular}
\end{table}

\begin{table}[t]
\caption{Errors in the solution covariance matrices, $|\cov_u - \cov_u^\star|$}
\label{tab:err_u_var}
\begin{tabular}{c|ccc|ccc}
    & \multicolumn{3}{c|}{Sparse} & \multicolumn{3}{c}{TT} \\ \hline
$p$ $\backslash$ $M$ & 10       & 20        & 30           & 10        & 20      & 30 \\ \hline
1                    & 9.49e-2  & 8.86e-2   & 9.67e-2      & 4.18e-2   & 2.80e-2 & 2.60e-2  \\
2                    & 3.46e-3  & 2.65e-3   & 3.34e-3      & 1.00e-4   & 1.31e-4 & 2.12e-4  \\
3                    & 1.65e-4  & 2.77e-4   & ---          & 4.48e-5   & 1.32e-4 & 2.14e-4  \\
4                    & 8.58e-5  & ---       & ---          & 6.28e-5   & 1.33e-4 & 1.11e-4  \\
\end{tabular}
\end{table}

\begin{figure}[t]
\centering
\caption{Errors in the solution covariance matrices with different $p$ (left), and versus the total CPU time in seconds (right). Number of PCE variables: $M=10$.}
\label{fig:err_u_var}
\includegraphics[width=.49\textwidth]{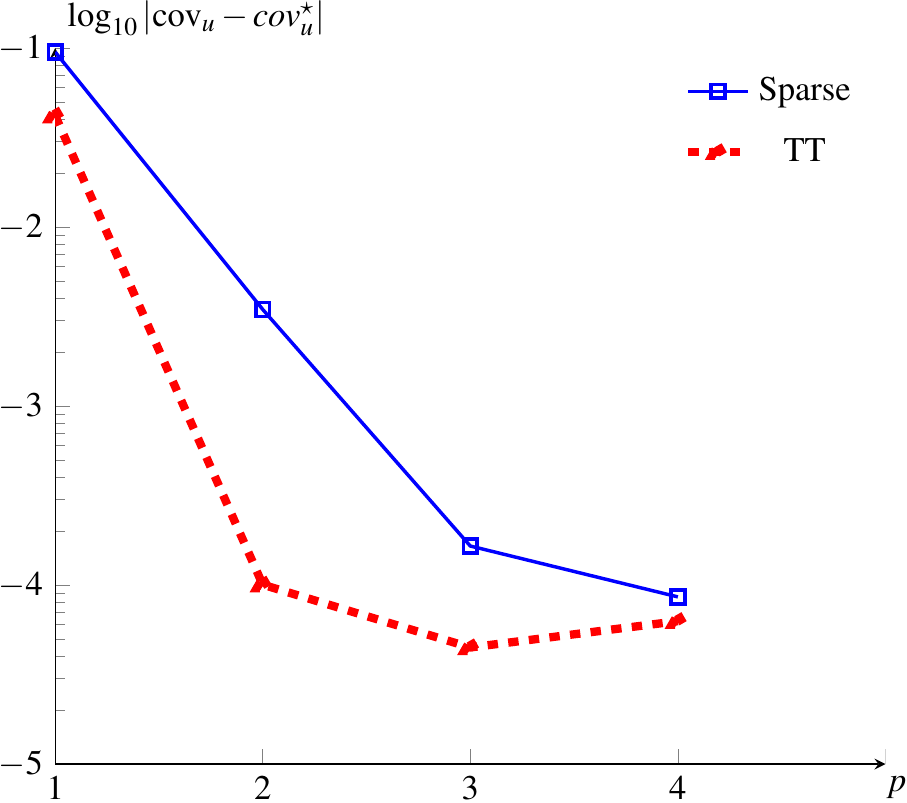} \hfill
\includegraphics[width=.49\textwidth]{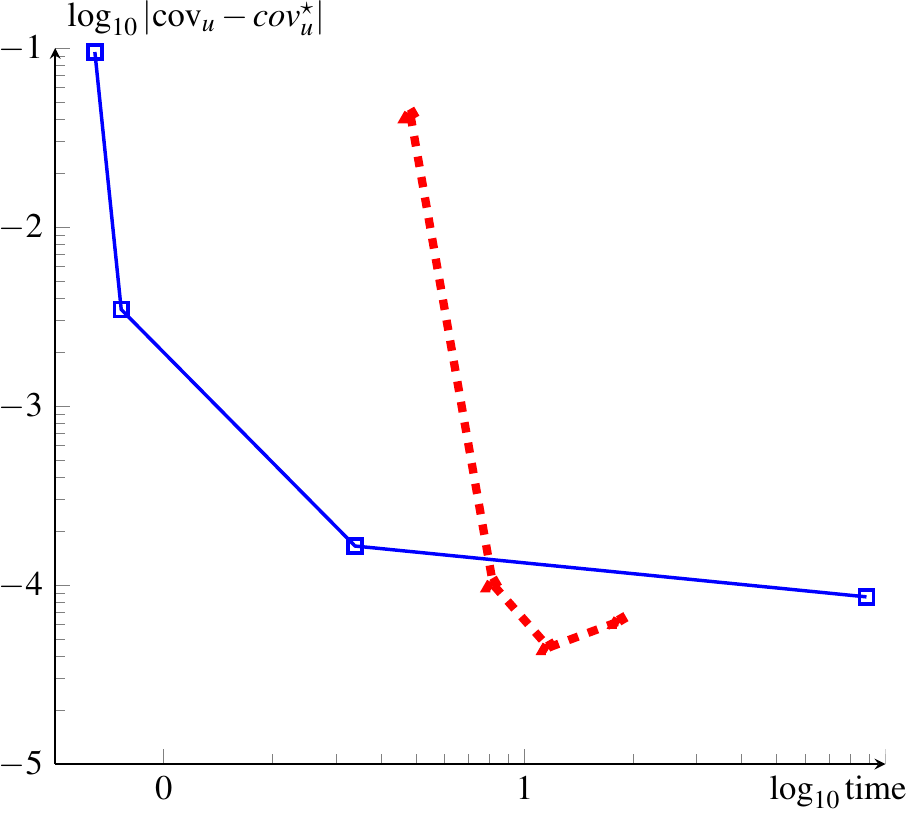}
\end{figure}

\begin{figure}[t]
\centering
\caption{Mean solution (left) and its variance (right). Number of PCE variables: $M=20$, $p=5$.}
\label{fig:u_space}
\includegraphics[width=.49\textwidth]{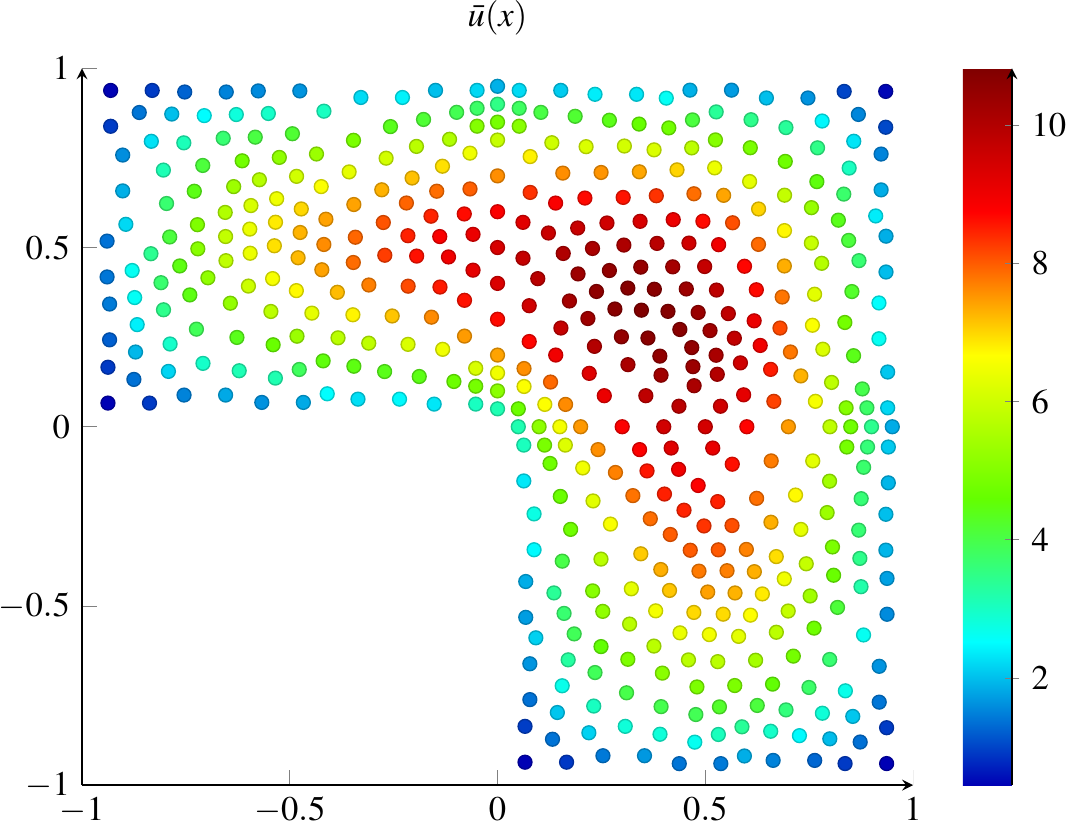} \hfill
\includegraphics[width=.49\textwidth]{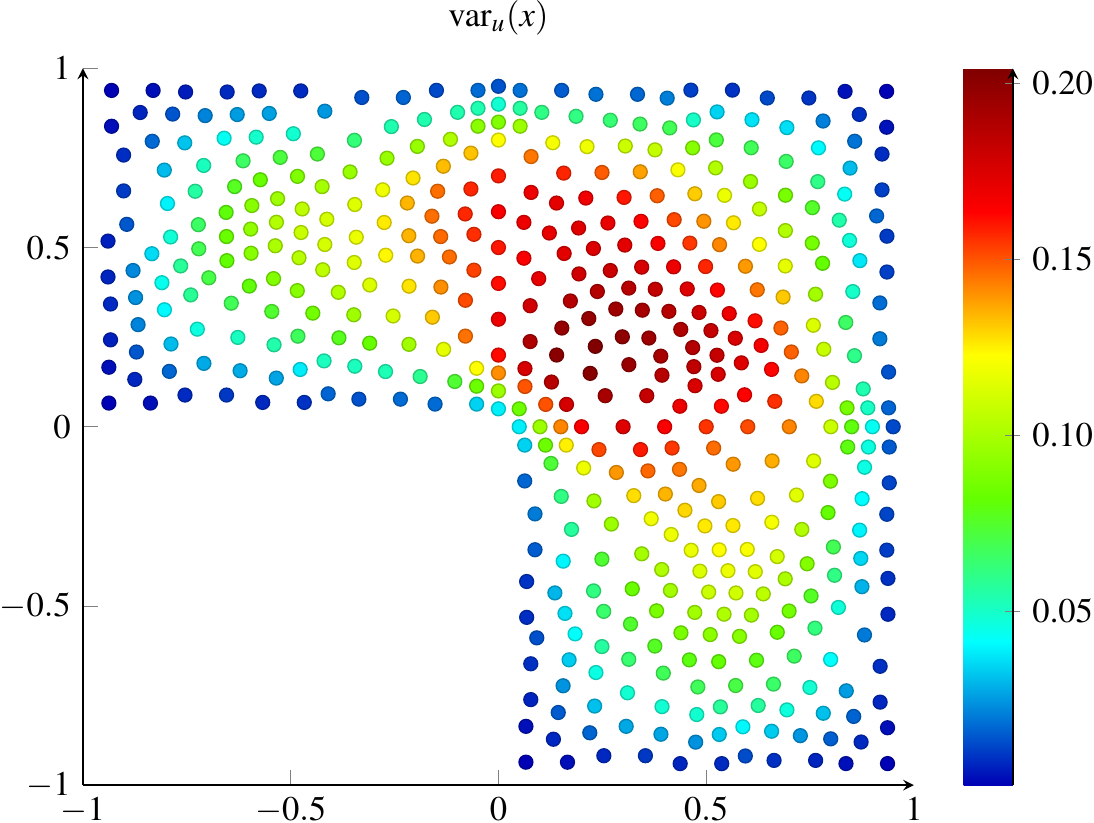}
\end{figure}

\section{Conclusion}

In this work we presented some new techniques for solving elliptic differential equation with uncertain coefficients.
We the applied tensor train (TT) data format to the PCE coefficients, which appear in the stochastic Galerkin approach.
The principal favor of the TT format comparing to e.g. the Canonical Polyadic (CP) decomposition is a stable quasi-optimal rank reduction procedure based on singular value decompositions.
The complexity scales as $\mathcal{O}(Mnr^3)$, i.e. is free from the curse of dimensionality, while the full accuracy control takes place.
We have extended the conception of the cross TT approximation using a few amount of samples to the PCE calculation task, and developed the new Block TT-Cross algorithm.

In Section \ref{subsec:levelSets} we demonstrated how the low-rank tensor representation of the solution can be used for post-processing.
We showed how to compute the mean value, covariance and level sets in the low-rank tensor train data format.

Traditional tools, e.g. the Stochastic Galerkin library (\emph{sglib}), generate random fields and solve stochastic PDEs in a sparse polynomial chaos.
In the considered numerical example, comparison of the sparse PCE and full PCE in the TT format is not so obvious:
the TT methods become preferable for high polynomial orders $p$, but otherwise the direct computation in a small sparse set may be incredibly fast.
This reflects well the ``curse of order'', taking place for the sparse set instead of the ``curse of dimensionality'' in the full set: the cardinality of the sparse set grows exponentially with $p$.

Relatively large $p$ (e.g. $4$--$5$ compared to the traditional $2$--$3$) may be necessary for high accuracies.
The TT approach scales linearly with $p$, and hence is suitable for such demands.
If we need to cast the PCE coefficients to the actual response surface \eqref{eq:alpha_to_theta} as a function on $\theta$, we may additionally employ the QTT approximation \cite{khor-qtt-2011}, to achieve the logarithmic compression w.r.t. the number of degrees of freedom in $\theta$.

The strong side of the TT methods is the easy calculation of the stochastic Galerkin operator.
This task involves only univariate manipulations with the TT format of the coefficient.
Interestingly, we can even compute the Galerkin operator exactly by preparing the coefficient with twice larger polynomial orders than those employed for the solution.
With $p$ below $10$, the TT storage of the operator allows us to forget about the sparsity issues, since the number of TT entries $\mathcal{O}(M p^2 r^2)$ is tractable.
This also means that other polynomial families, such as the Chebyshev or Laguerre, may be incorporated into the scheme freely.

To sum up, the TT formalism may be recommended for stochastic PDEs as a general tool: one introduces the same discretization levels for all variables and let the algorithms determine a quasi-optimal representation adaptively.
Nevertheless, many questions are still open.
Can we endow the solution scheme with more structure and obtain a more efficient algorithm?
Is there a better way to discretize stochastic fields than the KLE-PCE approach?
In the preliminary experiments, we have investigated only the simplest statistics, i.e. the mean and variance.
What quantities outlined in Section \ref{subsec:levelSets} are feasible in tensor product formats and how can they be computed?
We plan to investigate these issues in a future research.

\subsection*{Acknowledgment}
We would like to thank Dr. Elmar Zander for a very nice assistance and help in the usage of the Stochastic Galerkin library \emph{sglib}.\\
Alexander Litvinenko and his research work reported in this publication was supported by the King Abdullah University of Science and Technology (KAUST).\\
A part of this work was done during his stay at Technische Universit\"at Braunschweig and was supported by German DFG Project CODECS "Effective approaches and solution techniques for conditioning, robust design and control in the subsurface".
Sergey Dolgov was partially supported by RSCF grants 14-11-00806, 14-11-00659, RFBR grants 13-01-12061-ofi-m,
14-01-00804-A, and the Stipend of President of Russia at the Institute of Numerical Mathematics of Russian Academy of Sciences.


\begin{thebibliography}{10}

\bibitem{bebe-2000}
{\sc M.~Bebendorf}, \href {http://dx.doi.org/10.1007/pl00005410} {{\em
  Approximation of boundary element matrices}}, Numer. Mathem., 86 (2000),
  pp.~565--589.

\bibitem{bebe-aca-2011}
{\sc M.~Bebendorf}, \href {http://dx.doi.org/10.1007/s00365-010-9103-x} {{\em
  Adaptive cross approximation of multivariate functions}}, Constructive
  approximation, 34 (2011), pp.~149--179.

\bibitem{Sudret_sparsePCE}
{\sc G.~Blatman and B.~Sudret}, \href
  {http://dx.doi.org/http://dx.doi.org/10.1016/j.probengmech.2009.10.003} {{\em
  An adaptive algorithm to build up sparse polynomial chaos expansions for
  stochastic finite element analysis}}, Probabilistic Engineering Mechanics, 25
  (2010), pp.~183 -- 197.

\bibitem{crestaux-sapce-2009}
{\sc T.~Crestaux, O.~Le~Maitre, and J.~Martinez}, \href
  {http://dx.doi.org/10.1016/j.ress.2008.10.008} {{\em Polynomial chaos
  expansion for sensitivity analysis}}, Reliability Engineering $\&$ System
  Safety, 94 (2009), pp.~1161--1172.

\bibitem{tamen}
{\sc S.~Dolgov}, \href {https://github.com/dolgov/tamen} {{\em {tAMEn}}}.
\newblock https://github.com/dolgov/tamen.

\bibitem{d-tamen-2014}
{\sc S.~V. Dolgov}, \href {http://arxiv.org/abs/1403.8085} {{\em Alternating
  minimal energy approach to {ODEs} and conservation laws in tensor product
  formats}}, {arXiv} preprint 1403.8085, 2014.

\bibitem{dkos-eigb-2014}
{\sc S.~V. Dolgov, B.~Khoromskij, I.~V. Oseledets, and D.~V. Savostyanov},
  \href {http://dx.doi.org/10.1016/j.cpc.2013.12.017} {{\em Computation of
  extreme eigenvalues in higher dimensions using block tensor train format}},
  Computer Phys. Comm., 185 (2014), pp.~1207--1216.

\bibitem{ds-amr1-2013}
{\sc S.~V. Dolgov and D.~V. Savostyanov}, \href
  {http://arxiv.org/abs/1301.6068} {{\em Alternating minimal energy methods for
  linear systems in higher dimensions. {Part I}: {SPD} systems}}, {arXiv}
  preprint 1301.6068, 2013.

\bibitem{ds-amr2-2013}
{\sc S.~V. Dolgov and D.~V. Savostyanov}, \href
  {http://arxiv.org/abs/1304.1222} {{\em Alternating minimal energy methods for
  linear systems in higher dimensions. {Part II}: {Faster} algorithm and
  application to nonsymmetric systems}}, {arXiv} preprint 1304.1222, 2013.

\bibitem{Doostan2009}
{\sc A.~Doostan and G.~Iaccarino}, \href
  {http://dx.doi.org/10.1016/j.jcp.2009.03.006} {{\em A least-squares
  approximation of partial differential equations with high-dimensional random
  inputs}}, Journal of Computational Physics, 228 (2009), pp.~4332 -- 4345.

\bibitem{doostan-non-intrusive-2013}
{\sc A.~Doostan, A.~Validi, and G.~Iaccarino}, {\em Non-intrusive low-rank
  separated approximation of high-dimensional stochastic models}, Comput.
  Methods Appl. Mech. Engrg.,  (2013), pp.~42--55.

\bibitem{wahnert-stochgalerkin-2014}
{\sc M.~Espig, W.~Hackbusch, A.~Litvinenko, M.~H.G., and P.~W\"ahnert}, \href
  {http://dx.doi.org/10.1016/j.camwa.2012.10.008} {{\em Efficient low-rank
  approximation of the stochastic galerkin matrix in tensor formats}},
  Computers and Mathematics with Applications, 67 (2014), pp.~818--829.

\bibitem{litvinenko-spde-2013}
{\sc M.~Espig, W.~Hackbusch, A.~Litvinenko, H.~Matthies, and E.~Zander}, \href
  {http://dx.doi.org/10.1007/978-3-642-31703-3_2} {{\em Efficient analysis of
  high dimensional data in tensor formats}}, in Sparse Grids and Applications,
  Springer, 2013, pp.~31--56.

\bibitem{fannes-mps-1992}
{\sc M.~Fannes, B.~Nachtergaele, and R.~Werner}, \href
  {http://dx.doi.org/10.1007/BF02099178} {{\em Finitely correlated states on
  quantum spin chains}}, Communications in Mathematical Physics, 144 (1992),
  pp.~443--490.

\bibitem{gostz-maxvol-2010}
{\sc S.~A. Goreinov, I.~V. Oseledets, D.~V. Savostyanov, E.~E. Tyrtyshnikov,
  and N.~L. Zamarashkin}, {\em How to find a good submatrix}, in Matrix
  Methods: Theory, Algorithms, Applications, V.~Olshevsky and E.~Tyrtyshnikov,
  eds., World Scientific, Hackensack, NY, 2010, pp.~247--256.

\bibitem{gt-maxvol-2001}
{\sc S.~A. Goreinov and E.~E. Tyrtyshnikov}, {\em The maximal-volume concept in
  approximation by low-rank matrices}, Contemporary Mathematics, 208 (2001),
  pp.~47--51.

\bibitem{Grasedyck_Kressner}
{\sc L.~Grasedyck, D.~Kressner, and C.~Tobler}, {\em A literature survey of
  low-rank tensor approximation techniques}, GAMM-Mitteilungen, 36 (2013),
  pp.~53--78.

\bibitem{KhorSurvey}
{\sc W.~Hackbusch and B.~Khoromskij}, \href
  {http://dx.doi.org/10.1016/j.jco.2007.03.007} {{\em Tensor-product
  approximation to operators and functions in high dimensions}}, J. Complexity,
  23 (2007), pp.~697--714.

\bibitem{khkaz-lap-2012}
{\sc V.~A. Kazeev and B.~Khoromskij}, \href
  {http://dx.doi.org/10.1137/100820479} {{\em Low-rank explicit {QTT}
  representation of the {L}aplace operator and its inverse}}, SIAM J. Matrix
  Anal. Appl., 33 (2012), pp.~742--758.

\bibitem{KeeseDiss}
{\sc A.~Keese}, {\em Numerical solution of systems with stochastic
  uncertainties. {A} general purpose framework for stochastic finite elements},
  Ph.D. Thesis, TU Braunschweig, Germany,  (2004).

\bibitem{vekh-lattice-2014}
{\sc V.~Khoromskaia and B.~Khoromskij}, \href {http://arxiv.org/abs/1405.2270}
  {{\em Grid-based lattice summation of electrostatic potentials by assembled
  rank-structured tensor approximation}}, {arXiv} preprint 1405.2270, 2014.

\bibitem{Hcovariance}
{\sc B.~Khoromskij, A.~Litvinenko, and H.~Matthies}, \href
  {http://dx.doi.org/10.1007/s00607-008-0018-3} {{\em Application of
  hierarchical matrices for computing the {K}arhunen-{L}o{\`e}ve expansion}},
  Computing, 84 (2009), pp.~49--67.

\bibitem{khos-pde-2010}
{\sc B.~Khoromskij and I.~V. Oseledets}, \href
  {http://dx.doi.org/10.2478/cmam-2010-0023} {{\em {Quantics-TT} collocation
  approximation of parameter-dependent and stochastic elliptic {PDEs}}},
  Comput. Meth. Appl. Math., 10 (2010), pp.~376--394.

\bibitem{khos-qtt-2010}
{\sc B.~Khoromskij and I.~V. Oseledets}, \href
  {http://dx.doi.org/10.1515/rjnamm.2011.017} {{\em {QTT}-approximation of
  elliptic solution operators in higher dimensions}}, Rus. J. Numer. Anal.
  Math. Model., 26 (2011), pp.~303--322.

\bibitem{KhSch-Galerkin-SPDE-2011}
{\sc B.~Khoromskij and C.~Schwab}, \href {http://dx.doi.org/10.1137/100785715}
  {{\em Tensor-structured {Galerkin} approximation of parametric and stochastic
  elliptic {PDEs}}}, SIAM J. of Sci. Comp., 33 (2011), pp.~1--25.

\bibitem{khor-qtt-2011}
{\sc B.~N. Khoromskij}, \href {http://dx.doi.org/10.1007/s00365-011-9131-1}
  {{\em $\mathcal{O}(d \log n)$--{Quantics} approximation of {$N$--$d$} tensors
  in high-dimensional numerical modeling}}, Constr. Appr., 34 (2011),
  pp.~257--280.

\bibitem{tobkress-param-2011}
{\sc D.~Kressner and C.~Tobler}, \href {http://dx.doi.org/10.1137/100799010}
  {{\em Low-rank tensor {Krylov} subspace methods for parametrized linear
  systems}}, SIAM J. Matrix Anal. Appl., 32 (2011), pp.~273--290.

\bibitem{tobler-htdmrg-2011}
{\sc D.~Kressner and C.~Tobler}, \href
  {http://dx.doi.org/10.2478/cmam-2011-0020} {{\em Preconditioned low-rank
  methods for high-dimensional elliptic {PDE} eigenvalue problems}},
  Computational Methods in Applied Mathematics, 11 (2011), pp.~363--381.

\bibitem{Krosche_2010}
{\sc M.~Krosche and R.~Niekamp}, \href
  {http://www.digibib.tu-bs.de/?docid=00036351} {{\em {Low Rank Approximation
  in Spectral Stochastic Finite Element method with Solution Space Adaption}}},
  informatikbericht 2010-03, Technische Universit{\"{a}}t Braunschweig,
  Brunswick, 2010.

\bibitem{kunoth-2013}
{\sc A.~Kunoth and C.~Schwab}, \href {http://dx.doi.org/10.1137/110847597}
  {{\em Analytic regularity and {GPC} approximation for control problems
  constrained by linear parametric elliptic and parabolic {PDEs}}}, SIAM J
  Control and Optim., 51 (2013), pp.~2442--2471.

\bibitem{LeMatre10}
{\sc O.~P. Le~Ma{\^{\i}}tre and O.~M. Knio}, \href
  {http://dx.doi.org/10.1007/978-90-481-3520-2} {{\em Spectral methods for
  uncertainty quantification}}, Series: Scientific Computation, Springer, New
  York, 2010.

\bibitem{LitvSampling13}
{\sc A.~Litvinenko, H.~Matthies, and T.~A. El-Moselhy}, \href
  {http://dx.doi.org/10.1007/978-3-642-41095-6_27} {{\em Sampling and low-rank
  tensor approximation of the response surface}}, in Monte Carlo and
  Quasi-Monte Carlo Methods 2012, J.~Dick, F.~Y. Kuo, G.~W. Peters, and I.~H.
  Sloan, eds., vol.~65 of Springer Proceedings in Mathematics $\&$ Statistics,
  Springer Berlin Heidelberg, 2013, pp.~535--551.

\bibitem{Bieri2009}
{\sc M.~Marcel and C.~Schwab}, \href
  {http://dx.doi.org/10.1016/j.cma.2008.08.019} {{\em {Sparse high order FEM
  for elliptic sPDEs}}}, Computer Methods in Applied Mechanics and Engineering,
  198 (2009), pp.~1149--1170.

\bibitem{Matthies_encycl}
{\sc H.~Matthies}, {\em Uncertainty Quantification with Stochastic Finite
  Elements.}, John Wiley and Sons, Ltd, Chichester West Sussex, 2007.
\newblock Part 1. Fundamentals. Encyclopedia of Computational Mechanics.

\bibitem{matthies-galerkin-2005}
{\sc H.~Matthies and A.~Keese}, \href
  {http://dx.doi.org/10.1016/j.cma.2004.05.027} {{\em {G}alerkin methods for
  linear and nonlinear elliptic stochastic partial differential equations}},
  Computer Methods in Applied Mechanics and Engineering, 194 (2005),
  pp.~1295--1331.

\bibitem{UQLitvinenko12}
{\sc H.~Matthies, A.~Litvinenko, O.~Pajonk, B.~V. Rosi\'c, and E.~Zander},
  \href {http://dx.doi.org/10.1007/978-3-642-32677-6_9} {{\em Parametric and
  uncertainty computations with tensor product representations}}, in
  Uncertainty Quantification in Scientific Computing, A.~M. Dienstfrey and
  R.~F. Boisvert, eds., vol.~377 of IFIP Advances in Information and
  Communication Technology, Springer Berlin Heidelberg, 2012, pp.~139--150.

\bibitem{matthieszander-lowrank-2012}
{\sc H.~Matthies and E.~Zander}, \href
  {http://dx.doi.org/10.1016/j.laa.2011.04.017} {{\em Solving stochastic
  systems with low-rank tensor compression}}, Linear Algebra and its
  Applications, 436 (2012), pp.~3819--3838.

\bibitem{Nouy07}
{\sc A.~Nouy}, \href {http://dx.doi.org/10.1016/j.cma.2007.05.016} {{\em A
  generalized spectral decomposition technique to solve a class of linear
  stochastic partial differential equations}}, Comput. Methods Appl. Mech.
  Engrg., 196 (2007), pp.~4521--4537.

\bibitem{Nouy10}
{\sc A.~Nouy}, \href {http://dx.doi.org/10.1007/s11831-010-9054-1} {{\em Proper
  generalized decompositions and separated representations for the numerical
  solution of high dimensional stochastic problems}}, Arch. Comput. Methods
  Eng., 17 (2010), pp.~403--434.

\bibitem{Nouy2009}
{\sc A.~Nouy and O.~L. Ma{\^{\i}}tre}, \href
  {http://dx.doi.org/10.1016/j.jcp.2008.09.010} {{\em Generalized spectral
  decomposition for stochastic nonlinear problems}}, Journal of Computational
  Physics, 228 (2009), pp.~202--235.

\bibitem{LitvNowak13}
{\sc W.~Nowak and A.~Litvinenko}, \href
  {http://dx.doi.org/10.1007/s11004-013-9453-6} {{\em Kriging and spatial
  design accelerated by orders of magnitude: Combining low-rank covariance
  approximations with fft-techniques}}, Mathematical Geosciences, 45 (2013),
  pp.~411--435.

\bibitem{tt-toolbox}
{\sc I.~Oseledets, S.~Dolgov, V.~Kazeev, D.~Savostyanov, O.~Lebedeva,
  P.~Zhlobich, T.~Mach, and L.~Song}, \href
  {https://github.com/oseledets/TT-Toolbox} {{\em {TT-Toolbox}}}.
\newblock https://github.com/oseledets/TT-Toolbox.

\bibitem{osel-tt-2011}
{\sc I.~V. Oseledets}, \href {http://dx.doi.org/10.1137/090752286} {{\em
  Tensor-train decomposition}}, SIAM J. Sci. Comput., 33 (2011),
  pp.~2295--2317.

\bibitem{ost-latensor-2009}
{\sc I.~V. Oseledets, D.~V. Savostyanov, and E.~E. Tyrtyshnikov}, \href
  {http://dx.doi.org/10.1007/s00607-009-0047-6} {{\em Linear algebra for tensor
  problems}}, Computing, 85 (2009), pp.~169--188.

\bibitem{ot-tt-2009}
{\sc I.~V. Oseledets and E.~E. Tyrtyshnikov}, \href
  {http://dx.doi.org/10.1137/090748330} {{\em Breaking the curse of
  dimensionality, or how to use {SVD} in many dimensions}}, SIAM J. Sci.
  Comput., 31 (2009), pp.~3744--3759.

\bibitem{ot-ttcross-2010}
{\sc I.~V. Oseledets and E.~E. Tyrtyshnikov}, \href
  {http://dx.doi.org/10.1016/j.laa.2009.07.024} {{\em {TT-cross} approximation
  for multidimensional arrays}}, Linear Algebra Appl., 432 (2010), pp.~70--88.

\bibitem{Rosic2013}
{\sc B.~Rosi\'c, A.~Ku\v{c}erov\'a, J.~S\'ykora, O.~Pajonk, A.~Litvinenko, and
  H.~Matthies}, \href
  {http://dx.doi.org/http://dx.doi.org/10.1016/j.engstruct.2012.12.029} {{\em
  Parameter identification in a probabilistic setting}}, Engineering
  Structures, 50 (2013), pp.~179 -- 196.
\newblock Engineering Structures: Modelling and Computations (special issue
  IASS-IACM 2012).

\bibitem{sav-qott-2014}
{\sc D.~V. Savostyanov}, \href {http://dx.doi.org/10.1016/j.laa.2014.06.006}
  {{\em Quasioptimality of maximum--volume cross interpolation of tensors}},
  Linear Algebra Appl.,  (2014).

\bibitem{so-dmrgi-2011proc}
{\sc D.~V. Savostyanov and I.~V. Oseledets}, \href
  {http://dx.doi.org/10.1109/nDS.2011.6076873} {{\em Fast adaptive
  interpolation of multi-dimensional arrays in tensor train format}}, in
  Proceedings of 7th International Workshop on Multidimensional Systems (nDS),
  IEEE, 2011.

\bibitem{Schwab_sparse_tensors11}
{\sc C.~Schwab and C.~J. Gittelson}, \href
  {http://dx.doi.org/10.1017/S0962492911000055} {{\em Sparse tensor
  discretizations of high-dimensional parametric and stochastic {PDE}s}}, Acta
  Numer., 20 (2011), pp.~291--467.

\bibitem{SCHWAB2006}
{\sc C.~Schwab and R.~A. Todor}, \href
  {http://dx.doi.org/10.1016/j.jcp.2006.01.048} {{\em {Karhunen-Lo\`{e}ve
  approximation of random fields by generalized fast multipole methods}}},
  Journal of Computational Physics, 217 (2006), pp.~100--122.

\bibitem{sglib}
{\sc E.~Zander}, \href {http://github.com/ezander/sglib} {{\em Stochastic
  galerkin library:}}, http://github.com/ezander/sglib,  (2008).

\bibitem{ZanderDiss}
{\sc E.~K. Zander}, {\em Tensor Approximation Methods for Stochastic Problems},
  PhD thesis, Dissertation, Technische Universit\"at Braunschweig, ISBN:
  978-3-8440-1871-4, 2013.

\end{thebibliography}
\end{document}